\documentclass[12pt,reqno]{amsart}

\usepackage{amsmath,amsthm,amscd,amssymb, latexsym}

\usepackage[T1]{fontenc}
\usepackage{enumitem}
\setlist[enumerate,1]{label=\arabic*),font=\normalfont}

\usepackage[a4paper,lines=45,textwidth=15.3cm,heightrounded]{geometry}

\allowdisplaybreaks

\newtheorem{theorem}{Theorem}[section]
\newtheorem{proposition}[theorem]{Proposition}
\newtheorem{lemma}[theorem]{Lemma}
\newtheorem{corollary}[theorem]{Corollary}
\newtheorem{definition}[theorem]{Definition}
\numberwithin{equation}{section}

\theoremstyle{definition}
\newtheorem{remark}[theorem]{Remark}

\def\dom{\operatorname{dom}}
\def\ran{\operatorname{ran}}
\def\domf{\operatorname{D}}
\def\ranf{\operatorname{R}}

\def\sH{{\mathfrak H}}
\def\sN{{\mathfrak N}}
\def\sL{{\mathfrak L}}

\def\st{{\mathfrak t}}

\def\sT{{\mathfrak T}}

\def\RE{{\rm Re\,}}
\def\IM{{\rm Im\,}}

\def\Ker{{\rm Ker\,}}

\def\cD{{\mathcal D}}
\def\cF{{\mathcal F}}
\def\cG{{\mathcal G}}
\def\cH{{\mathcal H}}
\def\cL{{\mathcal L}}
\def\cQ{{\mathcal Q}}

\def\cP{{\mathcal P}}

\def\bL{{\mathbf L}}

\def\dR{{\mathbb R}}
\def\dC{{\mathbb C}}
\def\eS{\tilde{S}}
\def\eA{\tilde{A}}

\def\wt{\widetilde}

\newcommand{\bm}[1]{{\mathbf{#1}}}

\newcommand{\map}[3]{#1\colon\, #2\to #3}
\newcommand{\angles}[1]{\langle#1\rangle}

\newcommand\Rtwo{\mathbb{R}^2}
\newcommand\Ltwo{{L_2(\Rtwo)}}

\newcommand\intRtwo{\int_{\Rtwo}}

\def\uphar{{\upharpoonright\,}}

\newcommand\sRlim{\operatorname*{s-R-lim}}

\newcommand\kker{\operatorname{\mathbf{ker}}}
\newcommand\kkei{\operatorname{\mathbf{kei}}}

\begin{document}

\title[$M$-sectorial extensions]{On $M$-sectorial extensions\\ of sectorial operators}

\author{Yu. M. Arlinski\u\i}
\address{Department of Mathematical Analysis \\
East Ukrainian National Dahl University \\
prospect Sovetskii, 59-a \\
Severodonetsk, 93400 \\
Ukraine} \email{yury.arlinskii@gmail.com}

\author{ A. B. Popov}
\address{Department of Mathematical Analysis \\
East Ukrainian National Dahl University \\
prospect Sovetskii, 59-a \\
Severodonetsk, 93400 \\
Ukraine} \email{Andrey.B.Popov@gmail.com}



\date{}

\begin{abstract}
In our article~\cite{ArlPopMAccExt} description in terms of abstract
boundary conditions of all $m$-accretive extensions and their resolvents of a closed densely defined sectorial operator $S$ have been obtained. In particular, if
$\{\cH,\Gamma\}$ is a boundary pair of $S$, then there is a
bijective correspondence between all $m$-accretive extensions
$\eS$ of $S$ and all pairs $\angles{\bm{Z},X}$, where $\bm{Z}$ is
a $m$-accretive linear relation in $\cH$ and
$X:\dom(\bm{Z})\to\overline{\ran(S_{F})}$ is a linear operator such
that:
\[
\|Xe\|^2\leqslant\RE(\bm{Z}(e),e)_{\cH}\quad\forall e\in\dom(\bm{Z}).
\]
As is well known the operator $S$ admits at least one $m$-sectorial
extension, the Friedrichs extension. In this paper, assuming that
$S$ has non-unique $m$-sectorial extension, we established
additional conditions on a pair $\angles{\bm{Z}, X}$
 guaranteeing that corresponding $\eS$ is $m$-sectorial extension of $S$. As an application, all $m$-sectorial extensions of a nonnegative
symmetric operator in a planar model of two point interactions
are described.
\end{abstract}

\maketitle

\thispagestyle{empty}
\tableofcontents

\section*{Introduction}

Let $\sH$ be a complex Hilbert space with the inner product $(\cdot,\cdot)$. We use the symbols $\dom(T)$, $\ran(T)$, $\ker(T)$ for the
domain, the range, and the null-subspace of a linear operator $T$. The resolvent set of a linear operator $T$ is denoted by $\rho(T)$. The linear space of bounded operators acting
between Hilbert spaces $\sH_1$ and $\sH_2$ is denoted by $\bL(\sH_1,\sH_2)$
and the Banach algebra $\bL(\sH,\sH)$ by $\bL(\sH)$.
A linear operator $T$ in a complex Hilbert space $\sH$ is called
\textit{accretive} if its
 numerical range
\begin{equation*}
W(T)\stackrel{def}{=}\left\{(Tu,u),u\in\dom(T),\|u\|=1\right\}
\end{equation*}
is contained in the closed right half-plane, i.e.,
\[
\RE (Tu,u)\geqslant 0\;\mbox{for all} \;u\in\dom(T).
\]
An accretive operator $T$ is called \emph{maximal accretive} or
\emph{$m$-accretive}, if one of the following equivalent conditions
holds~\cite{Kato,Phillips,Ph3}:
\begin{enumerate}
\item $T$ is closed and has no accretive extensions in $\sH$;
\item resolvent set $\rho(T)$ contains a point from an open left half-plane;
\item $T$ is a closed densely defined operator and its adjoint $T^*$ is an accretive
operator;
\item the operator $-T$ generates one-parameter contractive semigroup $U(t)=\exp(-tT)$, $t\ge0$.
\end{enumerate}
One can prove the following equality for an arbitrary $m$-accretive
operator~$T$:
\begin{equation}\label{eqn:kerT}
\ker(T)=\ker(T^*).
\end{equation}

Let $\alpha \in [0,\pi/2)$. Denote by the $\Theta(\alpha)$ the sector in the complex plane
\[
\Theta \left( \alpha \right)\stackrel{def}{=}\left\{ z\in \mathbb{C}:\left| {\arg z}
\right|\leqslant \alpha\right\}.
\]

A linear operator $S$ is called \emph{sectorial} with
the vertex at the origin and the semi-angle $\alpha$~\cite{Kato} if
$
W(S)\subseteq \Theta \left( \alpha \right).
$
Clearly, $S$ is sectorial if and only if:
\[
\left| {\IM\left( {Sx,x} \right)} \right|\leqslant \tan\alpha\, \RE\left( {Sx,x} \right),
\]
for all $x\in \dom(S)$. In particular, if $\alpha =0$, then
$(Sx,x)\ge 0$ for all $x\in\dom(S)$, i.e., $S$ is symmetric and
nonnegative operator. \textit{In the sequel we will use the word
``sectorial'' only for sectorial operators and sectorial sesquilinear
forms with vertex at the origin. In addition, if semi-angle of
sectorial operator $S$ is $\alpha$ we will call $S$
$\alpha$-sectorial operator}. A linear operator $S$ is called
\emph{$m$-sectorial}, if it is sectorial and $m$-accretive. If $T$
is $m$-$\alpha$-sectorial operator and if $\gamma\in(\alpha,\pi/2)$
then
\begin{equation}
\label{estimres} \lambda\in\dC\setminus\Theta(\gamma)\Rightarrow
||(T-\lambda I)^{-1}||\le\cfrac{1}{|\lambda|\,\sin(\gamma-\alpha)},
\end{equation}
and the one-parameter semigroup $U(t)=\exp(-tT)$, $t\ge 0$, admits a
holomorphic contractive continuation into the interior of the sector
${\Theta(\pi/2-\alpha)}$ \cite{Kato}.

 It is well-known that there is a one-to-one correspondence
between closed densely defined sectorial forms and $m$-sectorial
operators. This correspondence is given by the First and the Second
Representations Theorems~\cite{Kato}. We will denote by $T[u,v]$ the
closed form associated with $m$-sectorial extension $T$ and by
$\cD[T]$ its domain.

In the present paper we continue to study $m$-accretive extensions of a densely defined closed sectorial operator $S$. It is well known~\cite{Kato}, that $S$ admits at least
one $m$-sectorial extension $S_F$, the Friedrichs extension,
which is associated with the closure of sesquilinear form
$(Sf,g)$, $f,g\in\dom(S)$.
 In~\cite{ArlMaxSectExt,ExtExtRel,MAccExtRest,Arl12,ArlAbsBounCond,LMS2012}, the
boundary triplets methods have been applied for a description of all
$m$-accretive, $m$-sectorial extensions, and their resolvents for sectorial operators $S$
satisfying condition
\begin{equation}\label{cond5.1}
\dom(S^*)\subseteq\domf[S_N],
\end{equation}
where $S_N$ is ``extremal'' $m$-sectorial extension of $S$, called the Kre\u\i
n-von Neumann extension~\cite{ArlMaxSectExt, ExtExtRel}. Such extension is
an analog of the ``soft'' (``the Kre\u\i n'', ``the
Kre\u\i n-von Neumann'') \textit{nonnegative} selfadjoint extension of a
nonnegative symmetric operator, discovered by M.G.~Kre\u\i n in
\cite{Kr1, Kr2}. Recall that $S$ is called nonnegative if $(Sf,f)\ge 0$ for all $f\in\dom(S)$. Observe that condition \eqref{cond5.1} holds true if for sectorial $S$ the equality $\dom(S^*_F)+\dom(S^*_N)=\dom(S^*)$ is satisfied. The  latter occurs, for instance, if $S$ is coercive, i.e., $\RE(Sf,f)\ge m ||f||^2$ for all $f\in \dom(S)$, where $m>0$.

 In our recent paper~\cite{ArlPopMAccExt} in the general case of an \textit{arbitrary} closed densely
defined sectorial operator $S$ we propose a new approach for the problem of parametrization of all $m$-accretive extensions. Our methods is applicable, in particular, for sectorial operator $S$, having a unique $m$-sectorial extension ($S_F=S_N$). In this paper, assuming $S_F\ne S_N$, we apply our method for a description of all $m$-sectorial extensions.

Let $A$ be a densely defined closed symmetric operator in $\sH$. Extensions $\tilde A$ of
$A$ possessing property
\[
A\subset \tilde A \subset A^*
\]
are called \textit{quasi-selfadjoint} (\textit{proper,
intermediate}) extensions of $S$. The problem of existence and
description of all quasi-selfadjoint m-accretive extensions of a
nonnegative symmetric operator via linear-fractional transformation
has been solved in~\cite{ArTs2} and via abstract boundary conditions
in~\cite{Mikh, Koch, Arl1988, DMTs2, DM2}. We refer on this matter
to the survey~\cite{ArTs2009} where one can find information about
various approaches to the extension problem of nonnegative symmetric
operators. Notice that in~\cite{CAOT2012}, developing the method proposed in~\cite{ArTs6}, an intrinsic
parametrization of domains of all m-accretive
and m-sectorial quasi-selfadjoint extensions of nonnegative $A$ have
been obtained.

In the present paper we use the approach of
\cite{ArlPopMAccExt} for such extensions. Applications to nonnegative symmetric operator in a planar model of
two-centers point interactions are given. In one-center point interaction planar model the corresponding nonnegative symmetric operator admits a unique nonnegative selfadjoint extension \cite{Adamyan}, \cite{GKMT}, hence, the Friedrichs extension is unique among all quasi-selfadjoint $m$-accretive extensions \cite{Ts1980} and  all $m$-sectorial extensions \cite{ArlMaxSectExt}. In our paper~\cite{ArlPopMAccExt} we described all $m$-accretive extensions for this case. In the case of two and more centers, the Friedrichs extension is a non-unique element of the set of all nonnegative selfadjoint extensions,  therefore, there are non-selfadjoint $m$-accretive quasi-selfadjoint extensions and $m$-sectorial extensions.

\section{Preliminaries}

\subsection{Sectorial forms and operators}
The basic definitions and results concerning ses\-qui\-linear forms can
be found in~\cite{Kato}. If $\tau$ is a closed densely defined
sectorial form in the Hilbert space $\sH$, then by the First
Representation Theorem~\cite{Kr1,Kato}, there exists a unique
$m$-sectorial operator $T$ in $\sH$ associated with $\tau$ in the
following sense: $(Tu,v)=\tau[u,v],$
for all $u\in\dom(T)$ and for all $v\in\dom(\tau)$. The adjoint
operator $T^*$ is associated with the adjoint form
$\tau^*[u,v]:=\overline{\tau[v,u]}$. The nonnegative selfadjoint
operator $T_R$ associated with the real part
$\tau_R[u,v]:=\left(\tau[u,v]+\tau^*[u,v]\right)/2$ of the form
$\tau$ and is called the \textit{real part} of $T$. According to the
Second Representation Theorem~\cite{Kato} the equality
$
\dom(\tau)=\dom(T^{1/2}_R).
$ holds.
Moreover,
\[
\tau[u,v]=((I+iG)T^{\frac{1}{2}}_Ru,T^{\frac{1}{2}}_Rv),\;u,v\in\dom(\tau),
\]
where $G$ is a bounded selfadjoint operator in the subspace
$\overline{\ran(T_R)}$ and $||G||\le \tan\alpha$ iff $\tau$ is
$\alpha$-sectorial. It follows that
\begin{equation}
\label{viraz}
\begin{aligned}
\dom(T)&=\{u\in\dom (\tau):(I+iG)T^{1/2}_Ru\in\dom(\tau)\},\\
Tu&=T^{1/2}_R(I+iG)T^{1/2}_Ru,\; u\in\dom (T).
\end{aligned}
\end{equation}

In the sequel we will use the following notations for a $m$-sectorial operator~$T$:
\[
\domf[T]\stackrel{def}{=}\dom(T^{1/2}_R),\;\ranf[T]\stackrel{def}{=}\ran (T^{1/2}_R).
\]

Also, for a $m$-sectorial operator $T$ we denote by
\[
\hat T=T\uphar\overline{\ran(T)},\quad
\hat T_R=T_R\uphar\overline{\ran(T)}.
\]
Equality~\eqref{eqn:kerT} yields that $\ker(\hat T)=\ker(\hat T_R)=\{0\}$. 
From~\eqref{viraz} it follows for $\lambda=-a+ib$, $a,b\in\dR$,
$a>0$ (see~\cite{Kato1961,ExtExtRel})
\[
\begin{gathered}
(T-\lambda I)^{-1}=(T_R+aI)^{-1/2}(I+iG(\lambda))^{-1}(T_R+aI)^{-1/2}, \\
G(\lambda)=T^{1/2}_R(T_R+aI)^{-1/2}GT^{1/2}_R(T_R+aI)^{-1/2}-b(T_R+aI)^{-1}.
\end{gathered}
\]
The latter equalities imply the following statement.
\begin{proposition}[\cite{ExtExtRel}]
If $T=T_R^{1/2}(I+iG)T_R^{1/2}$ is a $m$-$\alpha$-sectorial operator
in the Hilbert space $\sH$, and $\gamma\in(\alpha,\pi/2)$, then
\begin{multline}\label{eqn:RFormDef}
\ranf[T] = \left\{f\in \sH: \sup_{x\in\dom(T)}\frac{|(f,x)|^2}{\RE(Tx,x)}<\infty\right\}=\\
=\left\{f\in \sH: \lim_{\substack{\overline{\lambda\to
0},\\\lambda\in\dC\setminus\Theta(\alpha)}} \left|\left((T-\lambda
I)^{-1}f,f\right)\right|<\infty\right\};
\end{multline}
\begin{multline}\label{eq:RFSlim3}
\lim\limits_{\substack{\lambda\to 0,\\\lambda\in\dC\setminus\Theta(\gamma)}} \left((T-\lambda I)^{-1}f,g\right)= \hat T^{-1}[f,g]\\
=((I+iG)^{-1}\hat T^{-1/2}_Rf,\hat T^{-1/2}_Rg),\quad f,g\in\ranf[T];
\end{multline}
\begin{align}
\lim\limits_{\substack{\lambda\to
0,\\\lambda\in\dC\setminus\Theta(\gamma)}} T_R^{1/2}(T-\lambda
I)^{-1}T_R^{1/2}g=(I+iG)^{-1}g;\; g\in\domf[T]\ominus\ker(T).
\end{align}
\end{proposition}

\subsection{The Friedrichs and Kre\u\i n-von Neumann m-sectorial extensions}

Let $S$ be an $\alpha$-sectorial operator. It is well
known~\cite{Kato}, that the form $(Su,v)$, $u,v\in\dom(S)$ is
closable. We will denote by $S[u,v]$ its closure. The
domain of the form $S[u,v]$ is denoted by $\domf[S]$. With the
closed form $S[u,v]$ is associated the maximal $\alpha$-sectorial
operator $S_F$, which is called the \emph{Friedrichs extension} of
$S$~\cite{Kato}. So $\domf[S]=\domf[S_F]$ and $S_F[u,v]=S[u,v]$ for
all $u,v\in\domf[S]$. Let $S_{FR}$ be the real part of $S_F$. Clearly,
$\domf[S]=\dom(S^{1/2}_{FR})$. We will use the representations
\begin{align*}
S[u,v]&=S_F[u,v]=((I+iG_F)S^{1/2}_{FR}u,S^{1/2}_{FR}v),\; u,v\in\domf[S]=\dom (S^{1/2}_{FR}),\\
S_Ff&=S^{1/2}_{FR}(I+iG_F)S^{1/2}_{FR}f,\; f\in\dom(S_F),\\
S^*_Fg&=S^{1/2}_{FR}(I-iG_F)S^{1/2}_{FR}g,\; g\in\dom(S^*_F).
\end{align*}

It follows from the definition of the closure of the form $(Su,v)$, that
\begin{equation}\label{eqn7}\ranf[S_F]= \ran(S_{FR}^{1/2}) = \left\{f\in
\sH:
\sup_{\varphi\in\dom(S)}\frac{|(f,\varphi)|^2}{\RE(S\varphi,\varphi)}<\infty\right\}.
\end{equation}
In the case of a nonnegative symmetric operator $S$ ($\alpha=0$) it
was discovered by M.G. Kre\u{\i}n~\cite{Kr1} that the set of all its
nonnegative selfadjoint extensions has a minimal element (in the
sense of associated closed quadratic forms). This minimal element
$S_N$ is defined in~\cite{Kr1} by means of linear--fractional
transformation. Another (equivalent) definitions of $S_N$ are given
in~\cite{AndoNishio} and in~\cite{CS}. If $\alpha\ne 0$, then the
corresponding m-sectorial analog of such extremal extension also
exists~\cite{ArlMaxSectExt,ExtExtRel} and can be defined similarly, see~\cite{ArlMaxSectExt,ExtExtRel,LMS2012}.
We preserve the same notation $S_N$ and the name
\textit{Kre\u{\i}n-von Neumann extension} in the general case of non
necessarily symmetric sectorial operator $S$.
We notice that interesting applications of Kre\u{\i}n-von Neumann extension of nonnegative symmetric operator can be found in \cite{AGMShT}.

The domain of closed sesquilinear form associated with
Kre\u{\i}n-von Neumann extension of $\alpha$-sectorial operator $S$
$S$ is given by (see~\cite{ExtExtRel})
\begin{equation}\label{eqn8}
\domf[S_N] = \left\{u\in\sH:
\sup_{\varphi\in\dom(S)}\frac{|(u,S\varphi)|^2}{\RE(S\varphi,\varphi)}<\infty\right\}.
\end{equation}
This is an analog of the formula established by T.~Ando and
K.~Nishio in~\cite{AndoNishio} for the case of nonnegative symmetric
operator $S$ ($\alpha=0$).

Let
\[
\sN_\lambda\stackrel{def}{=}\sH\ominus\ran(S-\bar{\lambda}I)
\]
be the defect subspace of a linear operator $S$. If $S$ closed and
densely defined, then
 \[\sN_\lambda=\ker(S^*-\lambda I).\]
It is easy to see, that if $\eS$ is an extension of $S$ with nonempty
resolvent set, then for all $\lambda,z\in\rho(\eS^*)$
\begin{gather}
(\eS^*-\lambda
I)(\eS^*-zI)^{-1}\sN_\lambda=(I+(z-\lambda)(\eS^*-zI)^{-1})\sN_\lambda=\sN_z.
\end{gather}

Note, that from~\eqref{eqn7} and~\eqref{eqn8} it follows that
\begin{equation}
\label{raven11} \domf[S_N]\cap
\sN_\lambda=\ranf[S_F]\cap\sN_\lambda.
\end{equation}
For the operators $S_F$, $S_N$, and for an arbitrary $m$-sectorial
extension $\eS$ of $S$ the following relations are valid
\cite{ArlMaxSectExt, ExtExtRel}:
\[
\domf[S]\cap\sN_\lambda=\{0\},
\]
\begin{equation}
\label{eq:DSN}
\domf[S_N]=\domf[S]\dot+\left(\sN_\lambda\cap\domf[S_N]\right),\;\lambda\in\rho(S^*_F).
\end{equation}
\begin{equation}\label{eq:formseq}
\domf[S]\subseteq\domf[\eS]\subseteq\domf[S_N],\quad \ranf[S_N]\subseteq\ranf[\eS]\subseteq\ranf[S_F],
\end{equation}
\begin{gather}
\eS[f,v]=S_N[f,v]\quad \forall f\in\domf[S],v\in\domf[\eS],\\
S_N[f,v]=(f,S^*v)\quad \forall f\in\domf[S],v\in\dom(S^*)\cap\domf[S_N],\\
\dom(S^*_F)=\domf[S]\cap\dom(S^*).
\end{gather}
If $S$ is coercive, then
\[
\begin{array}{l}
\dom(S_N)=\dom(S)\dot+\ker(S^*),\; S_N\uphar\ker(S^*)=0,\\
\domf[S_N]=\domf[S]\dot+\ker(S^*).
\end{array}
\]
The operator $S$ has a unique $m$-sectorial extension if and only
if, for some $\lambda \in \rho({S_F^*})$ (then for all $\lambda \in
\rho({S_F^*}))$:
\begin{equation}
\label{uniq}
 \sup\limits_{x\in\dom(S)}\frac{|(f_\lambda,x)|^2}
 {\RE(Sx,x)}=\infty
\quad\forall f_\lambda\in\sN_\lambda\setminus\{0\}.
\end{equation}

From~\eqref{eqn7},~\eqref{eqn8}, and
~\eqref{uniq} 
it follows that
\begin{multline}
\label{eq:SectCond1} S_N\ne S_F\iff
\domf[S_N]\cap\sN_\lambda\neq\{0\}\\
\iff
\ranf[S_F]\cap\sN_\lambda\ne\{0\},\;\lambda\in\rho(S_F^*).
\end{multline}

Taking into account~\eqref{eq:SectCond1},~\eqref{eqn:RFormDef},
\eqref{eq:RFSlim3} we get for $\mu\in\dC\setminus\Theta(\alpha)$
\begin{equation}
\label{defmu} \varphi_\mu\in\sN_\mu\cap \domf[S_N]\iff
\lim_{\substack{\overline{\lambda\to
0},\\\lambda\in\dC\setminus\Theta(\alpha)}}
\left|\left((S^*_F-\lambda
I)^{-1}\varphi_\mu,\varphi_\mu\right)\right|<\infty,
\end{equation}
and for $\gamma\in(\alpha,\pi/2)$
\begin{multline*}
\lim_{\substack{\lambda\to
0,\\\lambda\in\dC\setminus\Theta(\gamma)}} \left((S^*_F-\lambda
I)^{-1}\varphi_\mu,\psi_\mu\right)=( (I-iG_F)^{-1}\hat
S^{-1/2}_{FR}\varphi_\mu,\hat S^{-1/2}_{FR}\psi_\mu),\\
 \varphi_\mu,\psi_\mu\in\sN_\mu\cap\domf[S_N].
\end{multline*}
 Fix $z\in\rho(S_F^*)$ and
 define a linear manifold $\sL$:
\begin{equation}
\label{linl} \sL\stackrel{def}{=}\domf[S]\dotplus\sN_z,\quad
z\in\rho(S_F^*).
\end{equation}
Then $\sL$ does not depend on the choice of $z\in\rho(S_F^*)$~\cite{ArlPopMAccExt} and, clearly, $\dom(S^*)\subset\sL$.
We will denote by $\cP_{z, F}$ and $\cP_z$ the skew projectors in $\sL$ onto
$\domf[S]$ and $\sN_z,$ corresponding to the decomposition
\eqref{linl}. If $z=i$, these projectors we will denote by $\cP_F$
and $\cP_i$, respectively.

Let us consider the form $\hat S_z[u,v]$ defined on the linear
manifold $\sL$
\[
\hat S_z[u,v]=S[\cP_{\bar{z},F}u,\cP_{\bar{z},F}v]-z(\cP_{\bar{z},F}u,\cP_{\bar{z},F}v),
\quad \forall z\in\dC\backslash\Theta(\alpha).
\]
The following relations have been established in~\cite{ExtExtRel}:
\begin{align*}
\domf[S_N]&=\left\{u\in\sL:\lim_{\substack{\overline{z\rightarrow
0}\\z\in\dC\backslash\Theta(\alpha)}}
\left|\hat S_z[u]\right |<\infty\right\},\\
S_N[u,v]&=\lim_{\substack{z\rightarrow
0\\z\in\dC\backslash\Theta(\gamma)}}\hat S_z[u,v],\qquad
u,v\in\domf[S_N],\quad
\gamma\in(\alpha;\pi/2),
\end{align*}
\begin{multline}
\label{expsnuv} S_N[u,v]=\Biggl((I+iG_F)\left(S^{1/2}_{FR}\cP_{z,
F}u+z(I-iG_F)^{-1}\hat
S^{-1/2}_{FR}\cP_zu\right),\\
\left(S^{1/2}_{FR}\cP_{z, F}v+z(I-iG_F)^{-1}\hat
S^{-1/2}_{FR}\cP_zv\right)\Biggr),\; u,v\in\domf[S_N].
\end{multline}

\subsection{Boundary triplets and abstract boundary conditions for quasi-selfadjoint extensions of nonnegative symmetric operator}
Let $A$ be a closed densely defined symmetric operator in $\sH$.
Recall the definition of a boundary triplet (boundary value space)~\cite{GG1} for $A^*$.
\begin{definition}
\label{boundtrip} A triplet $\left\{\cH,\Gamma_1,\Gamma_0\right\}$
is called boundary triplet of ${A}^*$ if $\cH$ is a Hilbert space
and $\Gamma_0,\Gamma_1$ are bounded linear operators from the
Hilbert space $H_+=\dom(S^*)$ with the graph norm into $\cH$ such
that the map $\vec \Gamma=\bigl<\Gamma_0,\Gamma_1\bigr>$ is a
surjection from $H_+$ onto $\cH^2=\cH\oplus\cH$ and the Green
identity holds:
\begin{equation}
\label{btr} \left(A^*f,g\right)-\left(f, A^*g\right)= \left(\Gamma_1
f,\Gamma_0 g\right)_\cH - \left(\Gamma_0 f,\Gamma_1 g\right)_\cH
\quad\forall f,\,g\in H_+.
\end{equation}
\end{definition}

In the sequel for descriptions of extensions in terms of the abstract boundary conditions the \textit{linear relations} will be used.
One can find basic notions, and properties related to these objects in, for instance,~\cite{Arens,RB,GG1,DM2,LMS2012}.
The formulas
\begin{equation}
\label{prop} \dom(\wt A)= \left\{u\in\dom(A^*):\vec \Gamma u\in
\wt{\bf T}\right\},\; \wt A=A^* \upharpoonright\dom(\wt A)
\end{equation}
give a one-to-one correspondence between all quasi-selfadjoint
extensions $\wt A$ of $A$ ($A\subset \wt A \subset A^*$) and all
linear relations $\wt{\bf T}$ in $\cH$. Moreover $\wt
A^*\leftrightarrow \wt{\bf T}^*$. Therefore, an extension $\wt A$ is
a selfadjoint one if and only if the relation $\wt{\bf T}$ is a selfadjoint in $\cH$.

As it was shown in~\cite{DM1,DM2} the operators $A_0,\; A_1$
defined as follows
$$A_k=A^*\uphar\Ker\Gamma_k,\; k=0,1$$
are mutually transversal selfadjoint extensions of $A$, i.e.,
\[
\dom(A^*)=\dom (A_0)+\dom(A_1).
\]
The function
$\Gamma_0(\lambda):=\left(\Gamma_0\uphar\sN_\lambda\right)^{-1}$
\cite{DM1} is the $\gamma$-field corresponding to $A_0$~\cite{KL,KL2}, i.e.,
\begin{gather*}
\ran(\Gamma_0(\lambda))=\sN_\lambda,\\
\Gamma_0(\lambda)=\Gamma(z)+(\lambda-z) (A_0-z I)^{-1}\Gamma_0(z).
\end{gather*}
Note that as a consequence of~\eqref{btr} one can obtain
the equality
\begin{equation}
\label{g*} \Gamma_0(\overline{\lambda})=\left(\Gamma_1(A_0-\lambda
I)^{-1}\right)^*.
\end{equation}
V.~Derkach and M.~Malamud~\cite{DM1,DM2} define the Weyl function
$M_0(\lambda)$ by the equality
\begin{equation}
\label{weyl} M_0(\lambda)=\Gamma_1\Gamma_0(\lambda).
\end{equation}
The Nevanlinna class operator valued function $M_0$ is Kre\u{\i}n-Langer $Q$-function \cite{KL,KL2}, and the following identity
\begin{equation}\label{eq:MlMz}
M_0(\lambda)-M_0(z)=(\lambda-z)\Gamma_0^*(\bar z)\Gamma_0(\lambda)
\end{equation}
holds.
In terms of boundary triplet the connection between a
quasi-selfadjoint extension $\wt A_{\wt {\bf T}}$ defined by
relations~\eqref{prop} and its resolvent is given by the Kre\u\i n resolvent formula
\begin{multline}
\label{res} \left(\wt A_{\wt {\bf T}}-\lambda I\right)^{-1}=
\bigl(A_0-\lambda I\bigr)^{-1}+\Gamma_0(\lambda) \left(\wt{\bf
T}-M_0(\lambda)\right)^{-1}\Gamma^*_0(\overline{\lambda}),\\
\lambda\in\rho(A_0)\cap\rho(\wt A_{\wt {\bf T}}).
\end{multline}

The following theorem has been established by V.~Derkach and
M.~Malamud (see~\cite{DM1,DM2, DMTs2,MM1}).
\begin{theorem}
\label{BTN} Let $A$ be a closed nonnegative symmetric operator and
let $\left\{\cH,\Gamma_1,\Gamma_0\right\}$ be a boundary triplet of
$A^*$ such that $A_0=A_F(=A^*\uphar\Ker \Gamma_0)$. Then $A$ has a non-unique
nonnegative selfadjoint extension if and only if
$$\cD_0=\left\{h\in\cH: \lim\limits_{x\uparrow
0}\left(M_0(x)h,h\right)_\cH < \infty\right\} \ne \{0\},$$ and the
quadratic form
$$\tau[h]=\lim\limits_{x\uparrow 0}\left(M_0(x)h,h\right)_\cH,
\;\cD[\tau]=\cD_0$$ is bounded from below. Define by $M_0(0)$ the
selfadjoint linear relation in $\cH$ associated with $\tau.$ Then
the Kre\u{\i}n-von Neumann extension $A_N$ can be defined by the
boundary condition
$$\dom(A_N)= \left\{u\in\dom(A^*):\bigl<\Gamma_0 u,\Gamma_1 u\bigr>
\in M_0(0)\right\}.$$ The relation $M_0(0)$ is also the strong
resolvent limit of $M_0(x)$ when $x\to -0$. Moreover, $A_0$ and
$A_N$ are disjoint iff $\overline{\cD_0}=\cH$ and transversal iff
$\cD_0=\cH$. In there is a one-to-one correspondence given
by~\eqref{prop} between $m$-accretive extensions $\wt A_{\wt{\bf
T}}$ and $m$-accretive linear relations $\wt{\bf T}$ satisfying the condition
\begin{equation}
\label{nonneg} \dom(\wt{\bf T})\subseteq\cD_0,\; \RE(\wt{\bf
T}x,x)\ge \tau[x],\; x\in\dom(\wt{\bf T}).
\end{equation}
The extension $\wt A_{\wt{\bf T}}$ is $m$-$\alpha$-sectorial iff the
form
$$ (\wt{\bf
T}x,y)-\tau[x,y]$$ is $\alpha$-sectorial.
\end{theorem}

\section{Abstract boundary conditions for $m$-accretive extensions of sectorial operators}

Next, we recall some definitions and results established
in~\cite{ArlPopMAccExt}. A sesquilinear form
$$\tau[u,v]\stackrel{def}{=}S_{FR}[\cP_{-1,F}u,\cP_{-1,F}v]+(\cP_{-1}u,\cP_{-1}v),\;
u,v\in\sL$$ is a nonnegative and closed~\cite{ArlPopMAccExt} in the
Hilbert space $\sH$. So, we can consider the linear manifold $\sL$
as a Hilbert space with the inner product
\[
(u,v)_\tau=\tau(u,v)+(u,v)_\sH.
\]

\begin{definition}[\cite{ArlPopMAccExt}]
\label{gam} A pair $\{\cH,\Gamma\}$ is called \emph{boundary pair}
of $S$, if $\cH$ is a Hilbert space and $\Gamma\in\bL(\sL,\cH)$ is
such that $\ker(\Gamma)=\domf[S]$, $\ran(\Gamma)=\cH$.
\end{definition}
Let
\[
\gamma(\lambda)=\left(\Gamma\uphar\sN_\lambda\right)^{-1},\;
\lambda\in\rho(S^*_F).
\]
Then $\gamma(\lambda)\in\bL(\cH,\sH)$ for all
$\lambda\in\rho(S^*_F)$. The operator-function $\gamma(\lambda)$ is
called \emph{$\gamma$-field} of the operator $S$ associated with the
boundary pair $\{\cH,\Gamma\}$. Clearly, $\gamma(\lambda)$ maps
$\cH$ onto $\sN_\lambda$.
Hence $S^*\gamma(\lambda)=\lambda\gamma(\lambda)$ and
\[
\ker(\gamma^*(\lambda))=\ran(S-\bar{\lambda}I)
\]
The following relations are valid:
\begin{gather}
\gamma(\lambda)=\gamma(z)+(\lambda-z)(S^*_F-\lambda I)^{-1}\gamma(z),\label{eq:gammalz1}\\
\cP_{\lambda, F}u=u-\gamma(\lambda)\Gamma u,\; u\in\sL,\notag\\
\cP_F\gamma(\lambda)e=(\lambda-i)(S_F^*-\lambda
I)^{-1}\gamma(i)e,\quad \cP_i\gamma(\lambda)e=\gamma(i)e, \;
e\in\cH.\notag
\end{gather}

Define on $\sL$ one more sesquilinear form $l[u,v]$:
\begin{equation}\label{eq:ldef}
l[u,v]
=S_F[\cP_Fu,\cP_Fv]-i(\cP_iu,\cP_Fv)-i(\cP_Fu,\cP_iv)-i(\cP_iu,\cP_iv).
\end{equation}

Due to the equality
\[
\RE l[u]=\RE S[\cP_F u]=\left\|S^{1/2}_{FR}\cP_F u\right\|^2,\;
u\in\sL,
\]
the form $l[u,v]$ is accretive. Moreover,
\[
\inf_{\varphi\in\domf[S]}\left\{\RE l[u-\varphi]\right\}=0,\quad\forall u\in\sL,
\]
and
$l[\varphi,v]=(\varphi,S^*v)$ for all $\varphi\in\domf[S],v\in\dom(S^*)$.

Relations~\eqref{expsnuv} and~\eqref{eq:ldef} imply the
following representation of the form $S_N[\cdot,\cdot]$:
\begin{multline}
\label{vazhno}
S_N[u,v]=l[u,v]\\
\qquad+ \left[i\left(\gamma(i)\Gamma u,\gamma(i)\Gamma
v\right)+\left((I-iG_F)^{-1}\hat S^{-1/2}_{FR}\gamma(i)\Gamma u,\hat
S^{-1/2}_{FR}\gamma(i)\Gamma v\right)\right]\\
\qquad\qquad+2i\left((I-iG_F)^{-1}\hat S^{-1/2}_{FR}\gamma(i)\Gamma
u, S^{1/2}_{FR}\cP_F v\right),\; u,v\in\domf[S_N].
\end{multline}

\begin{definition}[\cite{ArlPopMAccExt}]
\label{boundtip1} The triplet $\{\cH,G,\Gamma\}$ is called
\emph{boundary triplet for the operator $S^*$} if $\{\cH,\Gamma\}$
is a boundary pair for $S$ and
$\map{G}{\dom(S^*)}{\cH}$ is a linear operator such that the relation
\begin{equation}\label{eq:Gdef}
l^*[u,v]=(S^*u,v)-(Gu,\Gamma v)_\cH,\quad \forall u\in\dom(S^*),\;
\forall v\in \sL
\end{equation}
is valid.
\end{definition}

It is shown in~\cite{ArlPopMAccExt} that there exists a unique
operator $\map{G}{\dom(S^*)}{\cH}$ such that,~\eqref{eq:Gdef} holds
and, moreover,
\begin{equation*}\label{eq:G}
Gu=\gamma^*(i)(S^*-iI)u.
\end{equation*}

Next, we define operator-functions $\cQ(\lambda)\in\bL(\cH)$,
$\cG(\lambda)\in\bL(\sH,\cH)$, $\Phi(\lambda)\in\bL(\sH,\sH)$,
$q(\lambda)\in\bL(\cH,\sH)$, $\lambda\in\rho(S_F)$ associated with
the boundary triplet for the operator $S^*$, see~\cite{ArlPopMAccExt}:
\[\begin{gathered}
\cQ(\lambda)\stackrel{def}{=} G\gamma(\lambda),\quad q(\lambda)\stackrel{def}{=}\left(G(S_F^*-\bar{\lambda}I)^{-1}\right)^*,\\
\cG(\lambda)\stackrel{def}{=}\left(S_{FR}^{1/2}\cP_F\gamma(\bar{\lambda})\right)^*,\quad
\Phi(\lambda)\stackrel{def}{=}\left(S_{FR}^{1/2}(S_F^*-\bar{\lambda}I)^{-1}\right)^*.
\end{gathered}\]

The following identities are valid~\cite{ArlPopMAccExt}:
\begin{gather}
\cQ(\lambda)=\gamma^*(i)(S^*_F-iI)\gamma(\lambda)=(\lambda-i)\gamma^*(i)\gamma(\lambda),\label{QFU}\\
\Phi(\lambda)-\Phi(z)=(\lambda-z)(S_F-\lambda I)^{-1}\Phi(z)=(\lambda-z)(S_F-zI)^{-1}\Phi(\lambda),\notag\\
\cG(\lambda)-\cG(z)=(\lambda-z)\gamma^*(\bar{z})\Phi(\lambda),\notag\\
q(\lambda)-q(z)=(\lambda-z)(S_F-\lambda I)^{-1}q(z),\notag\\
\cQ(\lambda)-\cQ(z)=(\lambda-z)q^*(\bar{\lambda})\gamma(z)\notag.
\end{gather}
Observe that the function $\cQ(\lambda)$ is an analog of the Weyl
function~\eqref{weyl} corresponding to a boundary triplet of the
adjoint to a symmetric operator, while $q(\lambda)$ is an analog of
the function in~\eqref{g*}.

Let $L$ be a linear operator in $\sL$ defined as follows:
\begin{equation}
\label{OPERL}
\begin{aligned}
\dom(L) &= \dom(S_F) \dotplus \sN_i,\\
L(u_F + u_i) &= S_Fu_F - iu_i,\quad u_F\in \dom(S_F), u_i \in \sN_i.
\end{aligned}
\end{equation}
Then $L$ is closed, and
\begin{gather*}
(Lu, \varphi) = l[u, \varphi]\quad \forall u\in \dom(L), \varphi\in \domf[S],\\
\ker(L - \lambda I) = \ran(q(\lambda))\quad \forall \lambda\in \rho(S_F),\\
\dom(L) = \dom(S_F) \dotplus \ran(q(\lambda))\quad \forall \lambda\in \rho(S_F).
\end{gather*}

\begin{definition}[\cite{ArlPopMAccExt}]
\label{boundtip2} Let $S$ be a densely defined sectorial operator
and let $\{\cH, \Gamma\}$ be a boundary pair for $S$. A triplet
$\{\cH ,G_*, \Gamma\}$ is called a boundary triplet for $L$ if $G_*:
\dom(L)\to \cH$ is a linear operator such that
\[
l[u, v] = (Lu, v) - (G_*u, \Gamma v)_\cH,\;\forall u\in \dom(L),\;
\forall v \in \sL.
\]
\end{definition}

The operator $G_*$ is uniquely defined~\cite{ArlPopMAccExt} and,
moreover, for each $\lambda\in \rho(S_F)$
\begin{equation}\label{opg*}
\begin{aligned}
&G_*f = \gamma^*(\bar\lambda)(S_F - \lambda I)f,\; f\in \dom(S_F),\\
&G_*q(\lambda)e = \cQ^*(\bar\lambda)e,\; e\in \cH.
\end{aligned}
\end{equation}
Thus, given a boundary pair $\{\cH, \Gamma\}$ for an operator $S$,
the boundary triplets corresponding to it are $\{\cH, G, \Gamma\}$
for $S^*$ and $\{\cH ,G_*, \Gamma\}$ for $L$, and we have the
abstract Green formula
\[
(Lu, v) - (u, S^*v) = (G_*u, \Gamma v)_\cH - (\Gamma u, Gv)_\cH,\;
\forall u\in \dom(L),\;\forall v\in \dom(S^*).
\]
Let $\eS$ be an $m$-accretive extension of $S$. The following inclusions are established in~\cite{ArlPopMAccExt}:
\begin{equation}
\label{eq:ReS}
\begin{aligned}
\dom(\eS)&\subseteq \sL,\\
\eS u+\lambda \cP_{\lambda}u&\in\ran(S_{FR}^{1/2})(=\ranf[S_F]),
\;\lambda\in\rho(S^*_F).
\end{aligned}
\end{equation}
The next two theorems follow from~\eqref{eq:ReS}.

\begin{theorem}[\cite{ArlPopMAccExt}]\label{thm:main}
Let $S$ be a densely defined closed sectorial operator. Let
$\{\cH,\Gamma\}$ be a boundary pair for $S$ and $\{\cH,G,\Gamma\}$
be a corresponding boundary triplet for $S^*$. If $\eS$ is an
$m$-accretive extension of $S$, then there exist linear operators
\[Z:\dom(\tilde{S})\to\cH\text{ and
}X:\dom(X)=\Gamma\dom(\tilde S)\to\overline{\ran(S_{F})},\] such
that:
\begin{enumerate}
\item $\dom(S)\subseteq\ker(Z)$;
\item $(\tilde{S}u,v)=l[u,v]+(Zu,\Gamma v)_\cH+2(X\Gamma u,S_{FR}^{1/2}\cP_Fv)$, $\forall u\in\dom(\tilde{S})$, $v\in\sL$
\item $\bm{Z}=\{\langle\Gamma u,Zu\rangle,u\in\dom(\eS)\}$~--- is an $m$-accretive linear relation in~$\cH$;
\item $\|Xe\|^2\leqslant\RE(\bm{Z}(e),e)_{\cH}$ for all $e\in\dom(\bm{Z})=\Gamma\dom(\tilde{S})$
\end{enumerate}
\end{theorem}

\begin{theorem}[\cite{ArlPopMAccExt}]\label{thm:bijective}
There is a bijective correspondence between all $m$-accretive
extensions $\eS$ of $S$ and all pairs $\angles{\bm{Z},X}$, where
$\bm{Z}$ is an $m$-accretive linear relation in $\cH$ and
$X:\dom(\bm{Z})\to\overline{\ran(S_{F})}$ is a linear operator such that:
\begin{equation}\label{eq:XeZe}
\|Xe\|^2\leqslant\RE(\bm{Z}(e),e)_{\cH}\quad\forall e\in\dom(\bm{Z}).
\end{equation}
This correspondence is given by the boundary conditions for the
domain and the action of $\eS$ as follows: for all $\RE\lambda<0$
\begin{equation}\label{eq:LBoundS_F}
\begin{aligned}
\dom(\eS)&=\left\{u\in\sL: \begin{aligned}1)\;&u-(q(\lambda)-2\Phi(\lambda)X)\Gamma u\in\dom(S_F);\\
2)\;&G_*(u+2\Phi(\lambda)X\Gamma u)\in(\bm{Z}+2\cG(\lambda)X)\Gamma u
\end{aligned}
\right\},\\
\eS u&=S_F\bigl(u-(q(\lambda)-2\Phi(\lambda)X)\Gamma u\bigr)+\lambda\bigl(q(\lambda)-2\Phi(\lambda)X\bigr)\Gamma u.
\end{aligned}
\end{equation}
Set
\begin{equation}\label{eq:lrW}
\bm{W}(\lambda):=\bm{Z}-\cQ^*(\bar{\lambda})+2\cG(\lambda),\quad
\lambda\in\rho(S_F).
\end{equation}
Then
\begin{enumerate}
\item a number $\lambda\in\rho(S_F)$ is a regular point of $\eS$ if and only if
\[
\bm{W}^{-1}(\lambda)\in\bL(\cH),
\]
and,
\begin{equation}
(\eS-\lambda I)^{-1}=(S_F-\lambda I)^{-1}+(q(\lambda)-2\Phi(\lambda)X)\bm{W}^{-1}
(\lambda)\gamma^*(\bar{\lambda}),\label{eq:sp}
\end{equation}
\begin{equation}\label{domain}
\dom(\tilde S)=\Bigl(I+(q(\lambda)-2\Phi(\lambda)X)\bm{W}^{-1}(\lambda)\gamma^*(\bar\lambda)(S_F-\lambda I)\Bigr)\dom(S_F),
\end{equation}
\begin{equation}
\label{action}
\eS u=(S_F-\lambda I)f +\lambda u
\end{equation}
for
\begin{multline}\label{domain1}
u=\Bigl(I+(q(\lambda)-2\Phi(\lambda)X)\bm{W}^{-1}(\lambda)\gamma^*(\bar\lambda)(S_F-\lambda I)\Bigr)f,\; f\in\dom(S_F),
\end{multline}

\item a number $\lambda\in\rho(S_F)$ is an eigenvalue of $\eS$ if and only if
\[
\ker\left(\bm{W}(\lambda)\right)\neq\{0\},
\]
and,
\[
\ker(\eS-\lambda I)=(q(\lambda)-2\Phi(\lambda)X)\ker\left(\bm{W}(\lambda)\right).
\]
\end{enumerate}
\end{theorem}
\begin{remark}\label{mnim}
Relations~\eqref{eq:LBoundS_F} remain valid for all $\lambda\in\rho(\eS)\cap\rho(S_F).$
The resolvent formula~\eqref{eq:sp} is an analog of the
resolvent formula~\eqref{res}.
\end{remark}
Let $S$ be a densely defined closed sectorial operator. Define for
all $z\in\dC$, $\RE z\le 0$ a linear operator $S_z$ as follows~\cite{MAccExtRest,Arl12}:
\begin{equation}\label{eq:Szdef1}
\begin{aligned}
\dom(S_z)&=\dom(S)\dotplus\sN_z,\\
S_zh&=S\varphi-z\varphi_z,\; h=\varphi+\varphi_z\in\dom(S_z).
\end{aligned}
\end{equation}
\begin{proposition}[\cite{MAccExtRest,Arl12}]
The operator $S_z$ is $m$-accretive extension of~$S$.
\end{proposition}
\begin{proof}
Proposition has been proved in~\cite{MAccExtRest,Arl12} for $\RE z<0$. Let us prove the statement for $z=ix$, $x\in\dR$. Let $g=\varphi+\varphi_{ix}$, $\varphi\in\dom(S),$ $\varphi_{ix}\in \sN_{ix}$. Then
\begin{multline*}
(S_{ix}g,g)=(S\varphi-ix\varphi_{ix},\varphi+\varphi_{ix})\\
=(S\varphi,\varphi)-ix\|\varphi_{ix}\|^2-2i\IM(ix(\varphi_{ix},\varphi))
\end{multline*}
Hence $\RE(Sg,g)=\RE(S\varphi,\varphi)\ge 0$ for all
$g\in\dom(S_{ix})$. Furthermore, one can verify that
\[
\left\{\begin{aligned} &\dom ( S^*_{ix})=(S^*_F-ix I)^{-1}(S+ix I)\dom(S)\dot+\sN_{ix},\\
&S^*_{ix}\left((S^*_F-ix I)^{-1}(S+ix I)f\!+\!\varphi_{ix}\right)=S^*_F(S^*_F-ix I)^{-1}(S+ix I)f\!+\!ix\varphi_{ix},\\
&f\in\dom(S),\; \varphi_{ix}\in\sN_{ix}
\end{aligned}\right.
\]
and
\[
\RE(S^*_{ix}h,h)\!=\!\RE\left(S^*_F(S^*_F-ix I)^{-1}(S+ix I)f,(S^*_F-ix
I)^{-1}(S+ix I)f \right)\!\ge\! 0.
\]
for
\[
h=(S^*_F-ix I)^{-1}(S+ix I)f+\varphi_{ix},\; f\in\dom(S),\;
\varphi_{ix}\in\sN_{ix}.
\]
This means that $S^*_{ix}$ is accretive. Thus, $S_{ix}$ and
$S^*_{ix}$ are accretive. It follows that $S_{ix}$ is $m$-accretive.
\end{proof}
Note that in general from~\eqref{eq:Szdef1} it follows for $\RE z\le
0$ that
\begin{align*}
\dom(S^*_z)&=\left\{g\in\dom(S^*): (S^*+\bar{z} I)g\in \ran(S-\bar z I)\right\},\\
S^*_z&=S^*\uphar\dom(S^*_z).
\end{align*}
In addition, for the boundary operators in the boundary triplets in Definitions~\ref{boundtip1} and~\ref{boundtip2}, the equalities are valid
\[
\ker(G)=\dom (S^*_i),\; \ker (G_*)=\dom (S_i).
\]
\begin{remark}\label{vasa}
 It is proved in~\cite{Arl12,Arl2006} that
\begin{enumerate}
 \item for each $\gamma\in [0,\pi/2)$ the equalities are valid:
\[
\sRlim_{\substack{z\rightarrow 0\\-z\in\Theta(\gamma)}}S_z=S_N,\quad \sRlim_{\substack{z\rightarrow \infty\\-z\in\Theta(\gamma)}}S_z=S_F,
\]
where $\sRlim$ is the strong resolvent limit~\cite{Kato};
\item
the following conditions are equivalent:
\begin{enumerate}
\def\labelenumi{\rm (\roman{enumi})}
\item $S_z$ is $m$-sectorial operator for one (then for all) $z,\;\RE z<0$;
\item $\dom(S^*)\subset\dom (S_N)$, where $S_N$ is the Kre\u\i n--von Neumann extension of $S$.
\end{enumerate}
\end{enumerate}
\end{remark}
 Next we give expressions for pairs
$\angles{\bm{Z}_z,X_z}$ corresponding to $S_z$, $\RE z\le 0$ in
accordance with Theorem~\ref{thm:main}.
\begin{proposition}\label{prop:Sz} $\bm{Z}_z$ is the graph of the operator
$Z_z=-\cQ(z),$ $\dom(Z_z)=\cH$ and $X_z=-\cG^*(\bar{z}).$
In addition, for $u\in\dom(S_z)$, $v\in\sL$
\begin{equation}\label{eq:Szform}
(S_zu,v)=l[u,v]-(\cQ(z)\Gamma u,\Gamma v)_\cH-2(\cG^*(\bar{z})\Gamma
u,S_{FR}^{1/2}\cP_F v).
\end{equation}
\end{proposition}
\begin{proof}
Define for $u\in\dom({S_z})$
\begin{equation}\label{eq:Zdef}
\begin{aligned}
Z_zu&:=\gamma^*(i)({S_z}+iI)u,\\
M_zu&:=\frac{1}{2}\left(S_{FR}^{-1/2}({S_z}u+i\cP_iu)-(I+iG_F)S_{FR}^{1/2}\cP_Fu\right).
\end{aligned}
\end{equation}
Observe that from~\eqref{eq:Zdef} one obtains the inclusions
$\dom(S)\subseteq \ker (Z)$ and $\dom(S)\subset\ker (M_z)$. In addition,
due to definition of $\cL$~\eqref{linl}, Definition~\ref{gam} of a
boundary pair, and~\eqref{eq:Szdef1}, one obtaines the equality
\[
\Gamma\dom(S_z)=\cH.
\]
According to the proof of Theorem~\ref{thm:main}
(see~\cite{ArlPopMAccExt}), the relations
\[
\bm{Z}_z=\left\{\left<\Gamma u, Z_z u\right>,\; X_z\Gamma u=M_z u,\;u\in\dom(S_z)\right\}
\]
hold. Then, taking into account that $u=\gamma(z)\Gamma u$ and relations~\eqref{eq:Gdef},~\eqref{QFU}, \eqref{eq:Szdef1},
 we have
\begin{multline*}
Z_zu=\gamma^*(i)(S_z+iI)\gamma(z)\Gamma u=\gamma^*(i)(-z\gamma(z)\Gamma u+i\gamma(z)\Gamma u)=\\
=-(z-i)\gamma^*(i)\gamma(z)\Gamma u=-\cQ(z)\Gamma u.
\end{multline*}

Let $\Gamma u=e$, then $u=\varphi+\gamma(z)e$, $\varphi\in\dom(S)$,
and
\begin{multline*}
X_z\Gamma u=M_zu=M_z\gamma(z)e=\\
=\frac{1}{2}\left(S_{FR}^{-1/2}(S_z\gamma(z)e+iP_i\gamma(z)e)-(I+iG_F)S_{FR}^{1/2}\cP_F\gamma(z)e\right)=\\
=\frac{1}{2}\left(S_{FR}^{-1/2}(-z\gamma(z)e+i\gamma(i)e)-(I+iG_F)S_{FR}^{1/2}\cP_F\gamma(z)e\right)=\\
=\frac{1}{2}\bigl(S_{FR}^{-1/2}(-S^*\gamma(z)e+S^*\gamma(i)e)-(I+iG_F)S_{FR}^{1/2}\cP_F\gamma(z)e\bigr)=\\
=\frac{1}{2}\bigl(-S_{FR}^{-1/2}S_F^*\cP_F\gamma(z)e-(I+iG_F)S_{FR}^{1/2}\cP_F\gamma(z)e\bigr)=\\
=\frac{1}{2}\bigl(-(I-iG_F)S_{FR}^{1/2}\cP_F\gamma(z)e-(I+iG_F)\bigr)S_{FR}^{1/2}\cP_F\gamma(z)e=\\
=-S_{FR}^{1/2}\cP_F\gamma(z)e=-\cG^*(\bar{z})\Gamma u.
\end{multline*}

Equality~\eqref{eq:Szform} follows from Theorem~\ref{thm:main}.
\end{proof}

\section{$m$-sectorial extensions}

By Theorem~\ref{thm:bijective}, there is a bijective correspondence
between all $m$-accretive extensions $\eS$ of $S$ and all pairs
$\angles{\bm{Z},X}$ satisfying condition~\eqref{eq:XeZe}. Our main
goal is to establish additional conditions which guarantee that
 corresponding $m$-accretive extension $\eS$ is sectorial.

Next, we will need the following auxiliary result:
\begin{lemma}\label{lem:limT}
\begin{enumerate}
\item
If $T$ is a $m$-accretive operator and $\beta\in(0,\pi/2)$, then:
\begin{equation}
\label{predel1} \lim\limits_{\substack{z\to 0,\\\pi/2+\beta\le|\arg
z|\le \pi}}
 z(T-zI)^{-1}h=\left\{\begin{aligned}&-h,\quad &h\in {\ker(T)}\\
&0,&h\in\overline{\ran(T)}\end{aligned}\right..
\end{equation}
\item If $T$ is $m$-$\alpha$-sectorial and $\beta\in(\alpha,\pi/2)$, then
\begin{equation} \label{predel2}
\lim\limits_{\substack{z\to 0,\\ z\in\dC\setminus\Theta(\beta)}}
z(T-zI)^{-1}h=\left\{\begin{aligned}&-h,\quad &h\in {\ker(T)}\\
&0,&h\in\overline{\ran(T)}\end{aligned}\right..
\end{equation}
\end{enumerate}
\begin{proof}
1) Clearly
\[
h\in\ker(T)\Rightarrow(T-z I)^{-1}h=-\frac{h}{z}\quad\mbox{for all}\quad z\in\rho(T)\setminus\{0\}.
\]
Therefore
\[
\lim\limits_{\substack{z\to 0,\\\pi/2+\beta\le|\arg z|\le \pi}}
z(T-zI)^{-1}h=-h.
\]

Now let, $h\in\ran(T)$. Then $h=T\varphi$, $\varphi\in\dom(T)$ and
\begin{multline*}
z(T-zI)^{-1}h=z(T-zI)^{-1}T\varphi=\\
=z(T-zI)^{-1}(T-zI+zI)\varphi=z\varphi-z^2(T-zI)^{-1}\varphi.
\end{multline*}

Taking into account that
\[
\|(T-zI)^{-1}\|\leqslant\cfrac{1}{|\RE z|}, \;\RE z<0,
\]
 and $|\RE z|\ge
|z|\sin\beta$ for $\pi/2+\beta\le|\arg z|\le \pi$,
we get for all $\varphi\in\dom(T)$ that
\[
\lim\limits_{\substack{z\to 0,\\\pi/2+\beta\le|\arg z|\le
\pi}}z(T-zI)^{-1}T\varphi=0.
\]
Further, since $\ran(T)$ is dense in $\overline{\ran(T)}$ and
$$\|z(T-zI)^{-1}\|\leqslant\cfrac{1}{\sin \beta},\; \pi/2+\beta\le|\arg z|\le
\pi,
$$
then
\[
\lim\limits_{\substack{z\to 0,\\\pi/2+\beta\le|\arg z|\le \pi}}
 z(T-zI)^{-1}h
\]
for all $h\in\overline{\ran(T)}$. Thus~\eqref{predel1} is valid.

2) Relation~\eqref{predel2} follows from~\eqref{estimres}.
 \end{proof}
\end{lemma}
\begin{proposition}\label{prop:limzgammaz}
Let $S$ be a densely defined closed $\alpha$-sectorial operator,
$\gamma(z)$ its $\gamma$-field, corresponding to the boundary pair
$\{\cH,\Gamma\}$ of $S$. Suppose $S_F\ne S_N$. Then for all
$e\in\cH$ such that, $\gamma(\lambda)e\in\domf[S_N]$:
\[
\lim\limits_{\substack{z\to 0,\\
z\in\dC\setminus\Theta(\beta)}}z\gamma(z)e=0,
\]
where $\beta\in(0,\pi/2)$.
\end{proposition}
\begin{proof}
Let $\gamma(\lambda)e\in\domf[S_N]$. Since
$\domf[S_N]\cap\sN_\lambda=\ranf[S_F]\cap\sN_\lambda$, then
$\gamma(\lambda)e\in\ranf[S_F]$. Since
$\overline{\ranf[S_F]}=\overline{\ran(S_F)}=\overline{\ran(S_F^*)}$,
from Lemma~\ref{lem:limT} and~\eqref{eq:gammalz1} we have:
\[
\lim\limits_{\substack{z\to 0,\\
z\in\dC\setminus\Theta(\beta)}}z\gamma(z)e
 =\lim\limits_{\substack{z\to 0,\\
z\in\dC\setminus\Theta(\beta)}}\Bigl(z\gamma(\lambda)e+(z-\lambda)z(S^*_F-zI)^{-1}\gamma(\lambda)e\Bigr)=0.
\qedhere
\]
\end{proof}


\begin{theorem}\label{lem:QGlims}
Let $S$ be a densely defined closed sectorial operator, $\gamma(z)$
its $\gamma$-field, corresponding to the boundary pair
$\{\cH,\Gamma\}$ of $S$. Define a set in $\cH$:
\begin{equation}
\label{D0} \cD_0:=\left\{e\in\cH:\lim_{\substack{\overline{z\to
0},\\z\in\dC\setminus\Theta(\alpha)}}\left|(\cQ(z)e,e)_\cH\right|<\infty\right\}.
\end{equation}
Then
\[
\gamma(\mu)\cD_0=\sN_\mu\cap\domf[S_N].
\]
for all $\mu\in\dC\setminus\Theta(\alpha)$ and
\[
\cD_0=\Gamma\domf[S_N].
\]
Moreover, the following limits exist
\begin{gather*}
\Omega_0[e,g]:=-\lim_{\substack{z\rightarrow 0\\z\in\dC\setminus\Theta(\beta)}}(\cQ(z)e,g),\; e,g\in\cD_0,\\
X_0e:=-\lim\limits_{\substack{z\rightarrow
0\\z\in\dC\backslash\Theta(\beta)}}\cG^*(\bar z)e,\;e\in\cD_0,\;
\beta\in (\alpha,\pi/2),
\end{gather*}
and
\[
\begin{aligned}
\Omega_0[e,g]&=i\left(\gamma(i)e,\gamma(i)g\right)+\left((I-iG_F)^{-1}\hat S_{FR}^{-1/2}\gamma(i)e,
\hat S_{FR}^{-1/2}\gamma(i)g\right)=\\
&=i\left(\gamma(i)e,\gamma(i)g\right)+S_F^{*-1}\left[\gamma(i)e,\gamma(i)g\right],\; e,g\in\cD_0,
\\
X_0e&=i(I-iG_F)^{-1}\hat S_{FR}^{-1/2}\gamma(i)e,\; e\in\cD_0.
\end{aligned}
\]
\end{theorem}
\begin{proof}
Let $e\in\cH$. Then using~\eqref{eq:gammalz1} and~\eqref{QFU}
 we have for $z\in\dC\setminus\Theta(\alpha)$
\begin{multline*}
(\cQ(z)e,e)_\cH=(z-i)(\gamma(z)e,\gamma(i)e)\\
=(z-i)(\gamma(i)e+(z-i)((S^*_F-z I)^{-1}\gamma(i)e,\gamma(i)e)
\end{multline*}
Hence
\[
((S^*_F-z
I)^{-1}\gamma(i)e,\gamma(i)e)=-\cfrac{1}{z-i}(\gamma(i)e,\gamma(i)e)+\cfrac{1}{(z-i)^2}(\cQ(z)e,e)_\cH.
\]
The latter equality and~\eqref{eqn:RFormDef} yields
\begin{multline*}
\lim_{\substack{\overline{z\to
0},\\z\in\dC\setminus\Theta(\alpha)}}\left|(\cQ(z)e,e)_\cH\right|<\infty\\
\iff \lim_{\substack{\overline{z\to
0},\\z\in\dC\setminus\Theta(\alpha)}}\left|((S^*_F-z
I)^{-1}\gamma(i)e,\gamma(i)e)\right|<\infty\\
\iff \gamma(i)e\in\ranf[S_F]\cap\sN_i.
\end{multline*}
 Let $\cD_0$ be defined by~\eqref{D0}. Then, using~\eqref{raven11},
~\eqref{defmu},
and Corollary~\ref{prop:limzgammaz}, one obtains
\[
e\in\cD_0\iff \gamma(i)e\in\sN_i\cap\domf[S_N].
\]
Hence $\gamma(\mu)\cD_0=\sN_\mu\cap\domf[S_N]$ for all
$\mu\in\dC\setminus\Theta(\alpha)$. Observe that $\cD_0$ is a linear
manifold. Equality~\eqref{eq:DSN} yields that
$\Gamma\domf[S_N]=\cD_0$.

Notice that
the equality
\[
\gamma(z)=\gamma(i)+(z-i)(S^*_F-z I)^{-1}\gamma(i),
\]
the inclusion $\gamma(i)\cD_0\subseteq\overline{\ran (S^*_F)}$, and
applying Proposition~\ref{prop:limzgammaz} leads to
\[
\lim_{\substack{z\to 0,\\z\in\dC\setminus\Theta(\beta)}}
z\gamma(z)e=0,\;e\in\cD_0
\]
for $\beta\in (\alpha,\pi/2)$. Applying equality~\eqref{eq:RFSlim3},
we get the rest equalities in Theorem.
\end{proof}

Clearly the form $\Omega_0[e,g]$ can also be rewritten as follows:
\[
\Omega_0[e,g]=i\left(\gamma(i)e,\gamma(i)g\right)-i\left(X_0e,\hat
S_{FR}^{-1/2}\gamma(i)g\right),\; e,g\in\cD_0.
\]
Using expressions for $\Omega_0$
and $X_0$, by straightforward calculations one can deduce that
\begin{equation}
\label{ravomg}
\RE\Omega_0[e]=\|(I+iG_F)^{-1}S_{FR}^{-1/2}\gamma(i)e\|^2=\|X_0e\|^2,\;
e\in\cD_0.
\end{equation}
It follows that the sesquilinear form $\Omega_0[e,g]$ is accretive,
and, moreover, the form $\RE\Omega_0$ is closed in the Hilbert space
$\cH$. Observe that the form
 \begin{multline*}
\st_0[e,g]:=\Omega_0[e,g]-i(\gamma(i) e,\gamma(i)
g)=\left((I-iG_F)^{-1}\hat S_{FR}^{-1/2}\gamma(i)e, \hat
S_{FR}^{-1/2}\gamma(i)g\right)\\
=S^{*-1}_F[\gamma(i)e,\gamma(i)g],\; e,g\in\cD_0,
\end{multline*}
is closed and sectorial in $\cH$. Let the linear relation
$\mathbf{\sT}_0$ be associated with $\st_0$ by the First
Representation Theorem (see~\cite{RB} for nondensely defined closed sectorial forms). Then define
\[
\mathbf{Z}_0=\mathbf{\sT}_0+iP_{\overline\cD_0}\gamma^*(i)\gamma(i),
\]
where $P_{\overline\cD_0}$ is the orthogonal projection in $\cH$ onto the subspace $\overline\cD_0$.
The linear relation $\mathbf{Z}_0$ is $m$-accretive and associated
with the form $\Omega_0$ in the sense
\[
(\mathbf{Z}_0 e,g)_\cH=\Omega_0[e,g]\quad\mbox{for all} \quad
e\in\dom(\mathbf{Z}_0)\quad \mbox{and all}\quad g\in\cD_0.
\]


\begin{theorem}\label{thm:SNdef} Let $\{\cH,\Gamma\}$ be a
boundary pair of $S$. Then the pair $\angles{\bm{Z}_0,X_0}$
corresponds to the Kre\u{\i}n-von Neumann extension $S_N$ of the
operator $S$ in accordance with Theorem~\ref{thm:main}.
\end{theorem}
\begin{proof} It follows from~\eqref{vazhno} and from Theorem~\ref{lem:QGlims} that
\begin{equation}\label{eq:SNform1}
S_N[u,v]=l[u,v]+\Omega_0[\Gamma u,\Gamma v]+2(X_0\Gamma u,S_{FR}^{1/2}\cP_F v),\; u,v\in\domf[S_N].
\end{equation}
Let the pair $\angles{Z_N, X_N}$ corresponds to $S_N$ in accordance
with Theorem~\ref{thm:main}, $\dom(Z_N)=\dom(S_N),$
$\dom(X_N)=\Gamma \dom(S_N)$. Then
\begin{equation}\label{eq:SNform2}
(S_Nu,v)=l[u,v]+(Z_Nu,\Gamma v)_\cH+2(X_N\Gamma
u,S_{FR}^{1/2}\cP_Fv), \; u\in\dom (S_N),\; v\in\sL.
\end{equation}
Then~\eqref{eq:SNform1} and~\eqref{eq:SNform2} imply for
$v\in\domf[S]$ that
\[
(X_0\Gamma u,S_{FR}^{1/2} v)= (X_N\Gamma u,S_{FR}^{1/2}v).
\]
Hence $X_N=X_0\uphar\Gamma\dom(S_N)$. Further
\[
\Omega_0[\Gamma u,\Gamma v]=(Z_Nu,\Gamma v)_\cH,\; u\in\dom(S_N),\;
v\in\domf[S_N].
\]
Therefore, $m$-accretive linear relation
\[
\bm{Z_N}=\left\{\{\Gamma u, Z_N u\},\; u\in\dom (S_N)\right\}
\]
is associated with the form $\Omega_0$. It follows the equality
\[
{\bm Z_N}=\mathbf{Z}_0.\qedhere
\]
\end{proof}

\begin{remark}
If the set $\cD_0$ in Theorem~\ref{lem:QGlims} is trivial, then the
operator $S$ admits a unique $m$-sectorial extension, namely the
Friedrichs extension~$S_F$.
\end{remark}

Let
$$S_N[u,v]=\left((I+iG_N)S_{NR}^{1/2}u,S_{NR}^{1/2}v\right),\;
u,v\in\domf[S_N].$$ Since $S_N[u,v]=S_F[u,v]$, for all
$u,v\in\domf[S]$, there exists an isometry $U_F$ mapping
$\overline{\ran (S_F)}$ onto $\overline{\ran (S_N)}$ such that (see
\cite{ArlMaxSectExt,ExtExtRel})
\begin{gather*}
S_{NR}^{1/2}u=U_FS_{FR}^{1/2}u,\; u\in\domf[S],\\
G_NU_F=U_FG_F,\\
S_{NR}^{1/2}\varphi_\mu=\mu U_F(I-iG_F)^{-1}\hat
S^{-1/2}_{FR}\varphi_\mu,\; \varphi_\mu\in\mathfrak
N_\mu\cap\domf[S_N].
\end{gather*}
It follows that
\begin{equation}\label{eq:SNR}
S_{NR}^{1/2}u=U_FS_{FR}^{1/2}\cP_F u+U_FX_0\Gamma u,
\end{equation}

Description of all closed sesquilinear forms associated with
$m$-sectorial extensions of operator $S$ in the terms of boundary
pair has been obtained in~\cite{ArlMaxSectExt}.

\begin{definition}
[\cite{ArlMaxSectExt}]
\label{bp2} A pair $\{\cH',\Gamma'\}$ is called boundary
pair of the operator $S$, if $\cH'$ is a Hilbert space, and
$\Gamma':\domf[S_N]\to\cH'$ is a linear operator such that
$\ker(\Gamma')=\domf[S]$, $\ran(\Gamma')=\cH'$.
\end{definition}

Since $\domf[S]$ is a subspace in $\domf[S_N]$, the boundary pairs
$\{\cH',\Gamma'\}$ for operator $S$ exist.

\begin{theorem}[\cite{ArlMaxSectExt,ArlAbsBounCond}]\label{thm:3.6.3}
Let $\{\cH',\Gamma'\}$ be a boundary pair of the operator $S$ in the
sense of Definition~\ref{bp2}. Then the formula
\begin{equation}
\label{opisform}
\begin{aligned}
&\tilde{S}[u,v]=S_N[u,v]+\omega'[\Gamma' u, \Gamma' v]+2(X'\Gamma' u,S_{NR}^{1/2}v),\\
&u,v\in\domf[\tilde{S}]=\Gamma'^{-1}\domf[\omega']
\end{aligned}
\end{equation}
establish a bijective correspondence between all closed forms associated with $m$-sectorial extensions $\tilde{S}$ of $S$ and all pairs $\angles{\omega',X'}$, where
\begin{enumerate}
\item $\omega'$ is a closed and sectorial sesquilinear in the Hilbert space $\cH'$;
\item $X':\dom(\omega')\to\overline{\ran(S)}$ is a linear operator,
such that for some $\delta\in[0,1)$:
\[
\|X'e\|^2\leq \delta^2\RE\omega'[e],
\]
for all $e\in\dom(\omega')$.
\end{enumerate}
\end{theorem}

Let $\{\cH,\Gamma\}$ be a boundary pair of the operator $S$ in the
sense of Definition~\ref{gam}.
Set
\begin{equation}\label{inprod}
\begin{aligned}
&\cH'=\cD_0\;(=\dom(\Omega_0)),\\
&(e,g)_{\cH'}=(e,g)_\cH+\RE\Omega_0[e,g]=(e,g)_\cH+(X_0e,X_0g),\\
&\Gamma'=\Gamma\uphar\domf[S_N]=\Gamma\uphar\left(\domf[S]\dotplus\gamma(i)\cD_0\right).
\end{aligned}
\end{equation}
Then $\cH'$ is a Hilbert space w.r.t. the inner product $(\cdot,\cdot)_{\cH'}$ and $\{\cH',\Gamma'\}$ is boundary pair of the operator $S$ in the sense of Definition~\ref{bp2}.
Note that
\begin{enumerate}
\item
the operators $X_0$ and $\gamma(\lambda)$ are continuous from $\cH'$
into $\sH$,
\item the sesquilinear form $\Omega_0$ is continuous in $\cH'$.
\end{enumerate}

 Further, using Theorem~\ref{thm:main} and
representation~\eqref{eq:SNform1} for the form $S_N[u,v]$, we are
going to established additional conditions on the pairs
$\angles{\bm{Z}, X}$ that determine $m$-sectorial extensions of the
operator $S$ in accordance with Theorem~\ref{thm:bijective}.

\begin{theorem}\label{thm:sect}
Let $\{\cH,\Gamma\}$ be a boundary pair of $S$. Then the pair
$\angles{\bm{Z},X}$ determines an $m$-sectorial extension $\eS$ of $S$, see Theorem~\ref{thm:bijective} and Remark~\ref{mnim}, if
and only if the following conditions are fulfilled:
\begin{enumerate}
\item $\dom(\bm{Z})\subseteq \cD_0$;
\item the sesquilinear form
\begin{multline}\label{formaom}
{\omega}[e,g]=(\bm{Z}e,g)_\cH-\Omega_0[e,g]-2((X-X_0)e,X_0g)\\
= (\bm{Z}e,g)_\cH+\Omega^*_0[e,g]-2(X e,X_0g),\\
 e,g\in\dom({\bm{Z}})=\Gamma\dom(\eS)
\end{multline}
is sectorial and admits a closure in the Hilbert space $\cH'$;
\item $\|(X-X_0)e\|^2\leqslant \delta^2\RE {\omega}[e]$,\; $e\in\dom(\bm{Z})$ for some $\delta\in[0,1)$.
\end{enumerate}
Moreover, the closed sesquilinear form associated with $\tilde S$ is
given by
\begin{equation}
\label{formes}
\begin{gathered}
\tilde{S}[u,v]=
l[u,v]+{\bm{Z}}[\Gamma u,\Gamma v]+2(\overline{X}\Gamma u,
S^{1/2}_{FR}\cP_F v),\\
u,v\in\domf[\eS]=\Gamma^{-1}\dom(\overline{\omega}),
\end{gathered}
\end{equation}
where $\overline{X}$ is continuous extension of $X$ on the domain
$\dom(\overline{\omega})$ of the closure $\overline{\omega}$ of
$\omega$ and
\begin{equation}
\label{formaz}
{\bm{Z}}[e,g]:=\overline{\omega}[e,g]-\Omega^*_0[e,g]+2(\overline{X}e,X_0
g),\; e,g,\in\dom(\overline{\omega}).
\end{equation}
\end{theorem}
\begin{proof}
Let $\eS$ be an $m$-sectorial extension of $S$ determined by the
pair $\left<{\bm{Z}},X\right>$ in accordance with
Theorem~\ref{thm:main}. Note, that since $\eS$ is $m$-sectorial
extension of $S$, we have (see~\eqref{eq:formseq})
$\dom (\eS)\subset\domf[\eS]\subseteq\domf[S_N],$ and $\Gamma\dom(\eS)$
is a core of the linear manifold $\Gamma\domf[\eS]$.
Then
\[
(\eS u,v)=l[u,v]+(\bm{Z}\Gamma u,\Gamma v)_\cH+2(X\Gamma u,
S_{FR}^{1/2}\cP_F v), \; u,v\in\dom(\eS).
\]
Using~\eqref{eq:SNform1}, one obtains:
\begin{multline*}
(\eS u,v)=S_N[u,v] +(\bm{Z}\Gamma u,\Gamma v)_\cH-\Omega_0[\Gamma u,\Gamma v]\\
+2((X-X_0)\Gamma u,S_{FR}^{1/2}\cP_F v),\quad
u,v\in\dom(\eS).
\end{multline*}

From~\eqref{eq:SNR}
$
S_{FR}^{1/2}\cP_F v=U_F^*S_{NR}^{1/2}v-X_0\Gamma v.
$
Hence,
\begin{multline*}
(\eS u,v)=S_N[u,v]+ (\bm{Z}\Gamma u,\Gamma v)_\cH-\Omega_0[\Gamma u,\Gamma v]\\
-2((X-X_0)\Gamma u, X_0\Gamma v)+2(U_F(X-X_0)\Gamma
u,S_{NR}^{1/2}v)\\
=S_N[u,v]+ \omega[\Gamma u,\Gamma v]+2(U_F(X-X_0)\Gamma
u,S_{NR}^{1/2}v)\\
=S_N[u,v]+{\omega}[\Gamma' u, \Gamma'
v]+2(\tilde{X}\Gamma'u,S_{NR}^{1/2}v),\; u, v\in\dom (\eS),
\end{multline*}
where $\omega$ is given by~\eqref{formaom} and
$\tilde{X}=U_F(X-X_0)$.
 From Theorem
\ref{thm:3.6.3} it follows that $\omega$ is sectorial form,
$\dom(\omega)=\dom(\bm{Z})\subseteq \cD_0=\cH'$ and
$$||\tilde X e||^2=||(X-X_0)e||^2\le
\delta^2\RE\omega[e]$$ for all $e\in\dom(\bm{Z})$, where
$\delta\in[0,1)$. Moreover, the form $\omega$ admits closure
$\overline{\omega}$ in the Hilbert space $\cH'$, and $\tilde X$ can
be extended on $\dom\left(\overline{\omega}\right)$ by continuity as a linear
operator from $\dom\left(\overline{\omega}\right)$ with the inner product
\[
(e,g)_{\overline{\omega}}=(e,g)_{\cH'}+\RE\overline{\omega}[e,g].
\]
Since $X_0$ is continuous from $\cH'$ into $\sH$, the operator $X$
admits a continuation $\overline{X}$ on $\dom\left(\overline{\omega}\right)$. It
follows that the form ${\bm{Z}}$ given by~\eqref{formaz} is well
defined and the closed form $\eS[u,v]$ associated with $\eS$ is of
the form~\eqref{formes}.

Conversely, let conditions (1)--(3) of the theorem be fulfilled.
Denote by $\overline{\omega}$ the closure in the Hilbert space
$\cH'$ of the sesquilinear form $\omega$ given by
~\eqref{formaom}, and by $\overline{X'}$ the continuation of
the operator $\tilde{X}=U_F(X-X_0)$ on $\dom(\overline{\omega})$,
which exists due condition~(2). Then, by Theorem~\ref{thm:3.6.3},
the pair $\angles{\overline{\omega},\overline{X'}}$ determines
by~\eqref{opisform} a closed sectorial form $\eS[u,v]$ associated
with some $m$-sectorial
extension $\eS$ of the operator $S$.
\end{proof}

\begin{remark}
We can rewrite condition (3) of Theorem~\ref{thm:sect} in slightly
different form. Let us find the real part of the form
${\omega}[e,e]$. We have:
\[
{\omega}[e,e]=(\bm{Z}e,e)_\cH-\Omega_0[e,e]-2((X-X_0)e,X_0e).
\]
Using~\eqref{ravomg}, we obtain:
\begin{multline*}
\RE{\omega}[e,e]=\RE(\bm{Z}e,e)_\cH-\|X_0e\|^2+2\|X_0e\|^2-2\RE(Xe,X_0e)=\\
=\RE(\bm{Z}e,e)_\cH+\|X_0e\|^2-2\RE(Xe,X_0e)=\\
=\RE(\bm{Z}e,e)_\cH+\|(X-X_0)e\|^2-\|Xe\|^2.
\end{multline*}
Then the inequalities \[
\|(X-X_0)e\|^2\leqslant\delta^2\RE\omega[e]=\delta^2\left(\RE(\bm{Z}e,e)_\cH+\|(X-X_0)e\|^2-\|Xe\|^2\right)
\]
and $0\le \delta<1$ imply
\[
M\|(X-X_0)e\|^2\leqslant\RE(\bm{Z}e,e)_\cH-\|Xe\|^2,
\]
where $M=\cfrac{1-\delta^2}{\delta^2}>0$.

Thus, condition 3 can be rewritten as
\[\RE (\bm{Z}e,e)_\cH-\|Xe\|^2\geqslant M\|(X-X_0)e\|^2,\quad M>0.\]
\end{remark}

\section{Nonnegative symmetric operator and its quasi-selfadjoint $m$-accretive extensions}
In this section we will consider a densely defined closed
nonnegative symmetric operator $A$ and parameterize all its
quasi-selfadjoint $m$-accretive extensions in terms of abstract
boundary conditions. We will use a boundary pair and boundary
triplets defined in Definitions~\ref{gam},~\ref{boundtip1}, and
\ref{boundtip2}. In this case if $\{\cH,\Gamma\}$ is the boundary
pair for $A$ in the sense of Definition~\ref{gam}, then the
sesquilinear form $\Omega_0$ and the linear operator $X_0$ defined
on the linear manifold $\cD_0=\Gamma\domf[A_N]$ (see Theorem
\ref{lem:QGlims}) are of the form
\[
\begin{aligned}
\Omega_0[e,g]&=i\left(\gamma(i)e,\gamma(i)g\right)+\left(\hat
A_{F}^{-1/2}\gamma(i)e, \hat A_{F}^{-1/2}\gamma(i)g\right)
\\
X_0e&=i\hat A_{F}^{-1/2}\gamma(i)e,\; e,g\in\cD_0.
\end{aligned}
\]
In addition, from~\eqref{expsnuv} it follows that
\begin{multline}
\label{expsnuvsym} A_N[u,v]=\Biggl(\left(A^{1/2}_{F}\cP_{z,
F}u+z\hat A^{-1/2}_{F}\cP_zu\right), \left(A^{1/2}_{F}\cP_{z,
F}v+z\hat A^{-1/2}_{F}\cP_zv\right)\Biggr)\\
= \Biggl(\left(A^{1/2}_{F}(u-\gamma(z)\Gamma u)+z\hat
A^{-1/2}_{F}\gamma(z)\Gamma u\right),
\left(A^{1/2}_{F}(v-\gamma(z)\Gamma v)+z\hat
A^{-1/2}_{F}\gamma(z)\Gamma v\right)\Biggr),\\
u,v\in\domf[A_N]=\domf[A_F]\dot+(\sN_z\cap\ran(A^{1/2}_F))=\domf[A_F]\dot+\gamma(z)\cD_0.
\end{multline}
It is established in~\cite{Ar1} (see also~\cite{CAOT2012}) that the
following assertions are equivalent for $m$-accretive extension
$\eA$ of $A$:
\begin{enumerate}
\def\labelenumi{\rm (\roman{enumi})}
\item $A$ is quasi-selfadjoint extension;
\item $\dom(\eA)\subseteq \domf[A_N]$ and $\RE(\eA f,f)\ge A_N[f]$
for all $f\in\dom(\eA)$.
\end{enumerate}
Observe that the operator $L$ defined in~\eqref{OPERL} is of the
form
$$\dom(L)=\dom(A^*),\;Lu=A^*u-2iu_i,$$
where $u=u_F+u_i$, $u_F\in\dom(A_F),$ $u_i\in\sN_i$. If
$\{\cH,\Gamma\}$ is a boundary pair for $A$ (see Definition
\ref{gam}), then
\[
Lu=A^*u-2i\gamma(i)\Gamma u,\; u\in\dom(A^*).
\]

\begin{proposition}
Let $A$ be a closed densely defined nonnegative symmetric operator
in $\sH$ and let $\{\cH,\Gamma\}$ be its boundary pair (in the sense
of Definition~\ref{gam}). Assume $\cD_0\ne\{0\}$. Then a pair $\angles{\bm{Z},X}$ determines a
quasi-selfadjoint $m$-accretive extension $\eA$ of $A$ in
accordance with Theorem~\ref{thm:bijective} if and only if the
following conditions hold true
\begin{enumerate}
\item $\dom({\bm{Z}})\subseteq\cD_0$,
\item
$X=X_0\uphar\dom({\bm{Z}})=i\hat{A}_{F}^{-1/2}\gamma(i)\uphar\dom({\bm{Z}}).$
\end{enumerate}
\end{proposition}
\begin{proof}
Let $\eA$ be a quasi-selfadjoint $m$-accretive extension of the
operator $A$. Then $\dom(\eA)\subseteq\domf[A_N]$. By Theorem
\ref{thm:main} this implies the inclusion
$\dom({\bm{Z}})\subseteq\Gamma\domf[A_N]=\cD_0$. Taking into account
the decomposition $\dom(A^*)=\dom(A_F)\dotplus\sN_i$,
from~\eqref{eq:Zdef} for $\dom(\eA)\ni u=u_F+u_i$,
$u_F\in\dom(A_F)$, $u_i\in\sN_i$ we have
\begin{multline*}
X\Gamma u=Mu=\frac{1}{2}\left(\hat A_{FR}^{-1/2}(\eA u+i\cP_iu)-(I+iG_F)A_{FR}^{1/2}\cP_Fu\right)=\\
=\frac{1}{2}\left(\hat{A}_{F}^{-1/2}(A^*u+iu_i)-A_{F}^{1/2}u_F\right)=\frac{1}{2}\left(\hat{A}_{F}^{-1/2}(A_Fu_F+2iu_i)-A_{F}^{1/2}u_F\right)=\\
=i\hat{A}_{F}^{-1/2}\gamma(i)\Gamma u=X_0\Gamma u.
\end{multline*}

Now consider a pair $\angles{\bm{Z},X}$, where $\bm{Z}$ is
$m$-accretive linear relation in $\cH$ such that ({\bf{a}})
$\dom(\bm{Z})\subseteq\cD_0$ and ({\bf{b}}) $\RE(\bm{Z}e,e)_\cH\ge
||X_0e||^2$ for all $e\in\dom(\bm{Z})$. This pair determines an
$m$-accretive extension $\tilde A$. Let us prove that $\eA\subseteq
A^*$. Note that for all $u\in\sL$, $v\in\sH$
\begin{multline*}
\left(\Phi(\lambda)X_0\Gamma u, v\right)=i\left(\hat{A}_{F}^{-1/2}\gamma(i)\Gamma u, A_{F}^{1/2}(A_F-\bar{\lambda}I)^{-1}v\right)=\\
=i\left((A_F-\lambda I)^{-1}\gamma(i)\Gamma u, v\right).
\end{multline*}
So,
\begin{equation}\label{fxo}
\Phi(\lambda)X_0\Gamma u=i(A_F-\lambda I)^{-1}\gamma(i)\Gamma u\subset \dom(A_F).
\end{equation}
Using~\eqref{fxo} one gets
\begin{multline} \label{qg}
q(\lambda)-2\Phi(\lambda)X_0=\\
=\gamma(i)+(\lambda+i)(A_F-\lambda I)^{-1}\gamma(i)-2i(A_F-\lambda I)^{-1}\gamma(i)=\\
=\gamma(i)+(\lambda-i)(A_F-\lambda I)^{-1}\gamma(i)=\gamma(\lambda).
\end{multline}

From boundary conditions~\eqref{eq:LBoundS_F} for $u\in\sL$ we have:
\begin{multline*}
u\in\dom(\eA)\Rightarrow u-(q(\lambda)-2\Phi(\lambda)X_0)\Gamma u\in\dom(A_F)\\
\Rightarrow u-\gamma(\lambda)\Gamma u\in\dom(A_F),
\end{multline*}
and, therefore, $u\in\dom(A_F)\dotplus\sN_\lambda=\dom(A^*)$. Further, for $u=\cP_{\lambda,F}u+\cP_\lambda u$
\begin{multline*}
\eA u=A_F\bigl(u-(q(\lambda)-2\Phi(\lambda)X_0)\Gamma u\bigr)+\lambda\bigl(q(\lambda)-2\Phi(\lambda)X_0\bigr)\Gamma u=\\
=A_F\bigl(u-\gamma(\lambda)\Gamma u\bigr)+\lambda\gamma(\lambda)\Gamma u=\\
=A_F\cP_{\lambda,F}u+\lambda\cP_\lambda u=A^*(\cP_{\lambda,F}u+\cP_\lambda u).
\end{multline*}
So, $\eA\subseteq A^*$.
\end{proof}

\begin{theorem}\label{thm:properbijective}
Let $\{\cH,\Gamma\}$ and $\{\cH,G_*,\Gamma\}$ be a boundary pair for $A$
and the corresponding boundary triplet for $L$, see Definition
\ref{boundtip2}. Assume $\cD_0\ne\{0\}$. Then there is a bijective correspondence between all
$m$-accretive quasi-selfadjoint extensions $\eA$ of $A$ and all
$m$-accretive linear relations $\bm{Z}$ in $\cH$ such that
$\dom(\bm{Z})\subseteq\cD_0$ and:
\[\RE (\bm{Z}e,e)\geqslant \|\hat{A}_{F}^{-1/2}\gamma(i)e\|^2,\quad\forall e\in\dom(\bm{Z}).\]
This correspondence is given by 
\begin{equation}\label{eq:sLBoundS_F}
\begin{aligned}
\dom(\eA)&=\left\{u\in\dom(A^*): G_*u\in \bigl(\bm{Z}-2i\gamma^*(i)\gamma(i)\bigr)\Gamma u\right\},\\
\eA u&=A^*u.
\end{aligned}
\end{equation}
Moreover,
\begin{enumerate}
\item a number $\lambda\in\rho(A_F)$ is a regular point of $\eA$ if and only if
\[
\left(\bm{Z}-\frac{\lambda+i}{\lambda-i}\cQ(\lambda)\right)^{-1}\in\bL(\cH),
\]
and,
\begin{equation}
(\eA-\lambda I)^{-1}=(A_F-\lambda I)^{-1}+\gamma(\lambda)\left(\bm{Z}-\frac{\lambda+i}{\lambda-i}\cQ(\lambda)\right)^{-1}\gamma^*(\bar{\lambda});\label{eq:ssp}
\end{equation}

\item a number $\lambda\in\rho(A_F)$ is an eigenvalue of $\eA$ if and only if
\[
\ker\left(\bm{Z}-\frac{\lambda+i}{\lambda-i}\cQ(\lambda)\right)\neq\{0\},
\]
and,
\[
\ker(\eA-\lambda I)=\gamma(\lambda)\ker\left(\bm{Z}-\frac{\lambda+i}{\lambda-i}\cQ(\lambda)\right).
\]
\end{enumerate}
\end{theorem}
\begin{proof}
We will use~\eqref{eq:LBoundS_F}. Due to~\eqref{qg} the boundary
condition 1) in~\eqref{eq:LBoundS_F} is fulfilled. Let us transform
boundary condition 2). Due to~\eqref{opg*} we have for $\lambda\in
\rho(A_F)$
\[G_*(f + q(\lambda)e) =
\gamma^*(\bar\lambda)(A_F - \lambda I)f + \cQ^*(\bar\lambda)e,\;
f\in\dom(A_F),\;f\in\dom(A_F).
\]
So, we have
\begin{multline*}
G_*(u+2\Phi(\lambda)X\Gamma u)=G_*(u+(q(\lambda)-\gamma(\lambda))\Gamma u)=\\
=G_*(u+2i(A_F-\lambda I)^{-1}\gamma(i)\Gamma u)=\\
=G_*u+2i\gamma^*(\bar\lambda)\gamma(i)\Gamma u=\\
=\gamma^*(\bar\lambda)(A_F-\lambda I)\cP_{\lambda,F}u+\cQ^*(\bar\lambda)\Gamma u.
\end{multline*}

On the other hand,
\begin{multline*}
\bm{W}(\lambda)=\bm{Z}-\cQ^*(\bar{\lambda})+2\cG(\lambda)X_0\\
=\bm{Z}-(\lambda+i)\gamma^*(\bar{\lambda})\gamma(i)+2(\lambda+i)\gamma^*(i)\Phi(\lambda)X_0\\
=\bm{Z}-(\lambda+i)\bigl(\gamma^*(i)+(\lambda+i)\gamma^*(i)(A_F-\lambda I)^{-1}\bigr)\gamma(i)-2(\lambda+i)i\gamma^*(i)(A_F-\lambda I)^{-1}\gamma(i)\\
=\bm{Z}-(\lambda+i)\gamma^*(i)\bigl(I+(\lambda+i)(A_F-\lambda I)^{-1}-2i(A_F-\lambda I)^{-1}\bigr)\gamma(i)\\
=\bm{Z}-(\lambda+i)\gamma^*(i)\gamma(\lambda)=\bm{Z}-\frac{\lambda+i}{\lambda-i}\cQ(\lambda).
\end{multline*}
Then
\[
\bm{Z}+2\cG(\lambda)X_0=\bm{Z}+\cQ^*(\bar{\lambda})-\frac{\lambda+i}{\lambda-i}\cQ(\lambda).
\]
So, for the boundary condition~2) from from~\eqref{eq:LBoundS_F} one has
\begin{align*}
G_*u+2i\gamma^*(\bar\lambda)\gamma(i)\Gamma u&\in \left(\bm{Z}+\cQ^*(\bar{\lambda})-\frac{\lambda+i}{\lambda-i}\cQ(\lambda)\right)\Gamma u\\ \Longleftrightarrow G_*u&\in\left(\bm{Z}+\cQ^*(\bar{\lambda})-\frac{\lambda+i}{\lambda-i}\cQ(\lambda)-2i\gamma^*(\bar\lambda)\gamma(i)\right)\Gamma u\\
\Longleftrightarrow G_*u&\in\left(\bm{Z}+(\lambda+i)\gamma^*(\bar\lambda)\gamma(i)-\frac{\lambda+i}{\lambda-i}\cQ(\lambda)-2i\gamma^*(\bar\lambda)\gamma(i)\right)\Gamma u\\
\Longleftrightarrow G_*u&\in\left(\bm{Z}+\frac{\lambda-i}{\lambda+i}\cQ^*(\bar\lambda)-\frac{\lambda+i}{\lambda-i}\cQ(\lambda)\right)
\Gamma u.
\end{align*}
Further, using that $\cQ(\lambda)=(\lambda-i)\gamma^*(i)\gamma(i)$, we get
\begin{align*}
\left(\bm{Z}\right.&+\left.(\lambda-i)\gamma^*(\bar\lambda)\gamma(i)-(\lambda+i)\gamma^*(i)\gamma(\lambda)\right)\Gamma u\\
&=\left(\bm{Z}+(\lambda-i)\bigl(\gamma^*(i)+(\lambda+i)\gamma^*(i)(A_F-\lambda I)^{-1}\bigr)\gamma(i)\right.\\
&\quad-\left.(\lambda+i)\gamma^*(i)\bigl(\gamma(i)+(\lambda-i)(A_F-\lambda I)^{-1}\gamma(i)\bigr)\right)\Gamma u\\
&=\Bigl(\bm{Z}+\gamma^*(i)\bigl((\lambda-i)I+(\lambda^2+1)(A_F-\lambda I)^{-1}\bigr)\\
&\quad-\bigl((\lambda+i)I+(\lambda^2+1)(A_F-\lambda I)^{-1}\bigr)\Bigr)\gamma(i)\Gamma u\\
&=\bigl(\bm{Z}-2i\gamma^*(i)\gamma(i)\bigr)\Gamma u.\qedhere
\end{align*}
\end{proof}

\begin{remark} The boundary condition~\eqref{eq:sLBoundS_F} also can be written
for any $\lambda\in\rho(\eA)\cap\rho(A_F)$ as
\[
\dom(\eA)=\left\{u\in\dom(A^*): \gamma^*(\bar\lambda)(A_F-\lambda
I)(u-\gamma(\lambda)\Gamma u)\in\left(\bm{Z}-\frac{\lambda+i}{\lambda-i}\cQ(\lambda)\right)\Gamma u\right\},
\]
and
\[
\eA u=A^*u=A_F(u-\gamma(\lambda)\Gamma
u)+\lambda\gamma(\lambda)\Gamma u.
\]
\end{remark}

From Theorems~\ref{thm:sect},~\ref{thm:properbijective} we obtain
\begin{corollary}
Let $\bm{Z}$ be $m$-accretive linear relation, corresponding to a
quasi-selfadjoint $m$-accretive extension $\eA$ of $A$ by the
Theorem~\ref{thm:properbijective}. Then extension $\eA$ is a
sectorial (nonnegative) if and only if
\begin{enumerate}
\item $\dom(\bm{Z})\subseteq\cH'(= \dom(\Omega_0)=\cD_0)$;
\item the form $\tilde{\omega}[e,g]=(\bm{Z}e,g)_\cH-\Omega_0[e,g]$ is sectorial
(nonnegative).
\end{enumerate}
\end{corollary}

\begin{remark}\label{closure} The form $\tilde{\omega}$ admits a closure in the
Hilbert space $\cH'$ defined by~\eqref{inprod}. Actually, since
$\tilde{\omega}[e,g]=(\bm{Z}e,g)_\cH-\Omega_0[e,g]$ is sectorial,
the form
\[
\eta[e,f]=(\bm{Z}e,g)_\cH-i(\gamma(i)e,\gamma(i)f),\;
e,f\in\cH'(=\cD_0)
\]
is sectorial as well. If
\begin{gather*}
\lim\limits_{n\to \infty}e_n=0 \quad\mbox{in}\quad \cH',\\
\lim\limits_{m,n\to\infty}\tilde{\omega}[e_n-e_m]=0,
\end{gather*}
then
\begin{gather*}
\lim\limits_{n\to\infty}e_n=0\quad\mbox{in}\quad\cH,\quad\lim\limits_{n\to\infty}
\RE\Omega[e_n]=\lim\limits_{n\to\infty}||X_0e_n||^2=0,\\
\lim\limits_{n\to\infty}\gamma(i)e_n=0 \quad\mbox{in}\quad\sH.
\end{gather*}
Since linear relation $\bm{Z}$ is $m$-accretive and $\bm{Z}-i\gamma^*(i)\gamma(i)$
is sectorial, we get\linebreak $\lim\limits_{n\to\infty}(\bm{Z}e_n,e_n)_\cH=0$ (see~\cite{Kato}).
\end{remark}

Next we will find relationships between
\begin{itemize}
\item
 a boundary
triplet $\left\{\cH,\Gamma_1,\Gamma_0\right\}$ for ${A}^*$ given by
Definition~\ref{boundtrip} and boundary triplets
$\left\{\cH,G,\Gamma\right\}$, $\left\{\cH,G_*,\Gamma\right\}$ of Definitions~\ref{boundtip1} and~\ref{boundtip2};
\item parameterizations of quasi-selfadjoint $m$-accretive extensions
given by Theorem~\ref{BTN} and Theorem~\ref{thm:properbijective}.
\end{itemize}

Let $\left\{\cH,\Gamma_1,\Gamma_0\right\}$ be a boundary triplet of
${A}^*$ (see Definition~\ref{boundtrip}) such that $\ker(\Gamma_0)=\dom(A_F)$. Then
\begin{enumerate}
\item since $\dom(A_F)$ is a core of $\domf[A]$ and $\ker(\Gamma_0)=\dom(A_F)$,
 we can define a
boundary pair $\{\cH,\overline{\Gamma}_0\}$ where
$\overline{\Gamma}_0$ is a continuation of $\Gamma_0$ onto
$\sL=\domf[A]\dotplus\sN_i$ from $\dom(A^*)=\dom(A_F)\dotplus\sN_i$;
\item it follows that
$$\gamma(\lambda)=\left(\overline{\Gamma}_0\uphar\sN_\lambda\right)^{-1}=
\Gamma_0(\lambda);$$
\item
because relation~\eqref{eq:MlMz} can be rewritten as
\[
M_0(\lambda)-M_0(z)=(\lambda-z)\gamma^*(\bar z)\gamma(\lambda),
\]
using~\eqref{QFU}, one gets
\[
\cQ(\lambda)=(\lambda-i)\gamma^*(i)\gamma(\lambda)
=\frac{\lambda-i}{\lambda+i}(M_0(\lambda)-M_0(-i));
\]
so,
\begin{equation}
\label{mq}
M_0(\lambda)-M_0(-i)=\frac{\lambda+i}{\lambda-i}\cQ(\lambda);
\end{equation}
\item equation~\eqref{mq} yields that the linear manifolds $\cD_0$ in Theorems
\ref{BTN} and Theorem~\ref{lem:QGlims} coincide and
\[
\tau[h,g]=(M_0(-i)h,g)_\cH+\Omega_0[h,g],\; h,g\in\cD_0;
\]
\item comparing resolvent formulas~\eqref{res} and~\eqref{eq:ssp} we
 get that the linear relation $\bm{Z}$ from Theorem~\ref{thm:properbijective} and the linear relation $\wt{\bf T}$
from Theorem~\ref{BTN} (see~\eqref{prop}, \eqref{nonneg}) are
connected by the equality
\begin{equation}\label{eq:ZT}
\bm{Z}=\wt{\bf T}-M(-i).
\end{equation}
\end{enumerate}

\begin{proposition}
Let $\{\cH,\Gamma\}$ be a boundary pair for nonnegative symmetric operator $A$. Let $\eA$ be a quasi-selfadjoint $m$-accretive extension of $A$ and let $\bm{Z}$ be the corresponding linear relation in $\cH$ (see Theorem~\ref{thm:properbijective}). Then
\[
\bm{Z}^*+2i\gamma^*(i)\gamma(i)
\]
corresponds to the adjoint extension $\eA^*$.
\end{proposition}
\begin{proof}
The proof it easy, if we recall that to the adjoint extension $\eA^*$ corresponds the adjoint linear relation $\widetilde{\bm{T}}^*$. Since
\[
\wt{\bf T}=\bm{Z}+M(-i).
\]
Then
\[
\wt{\bf T}^*=\bm{Z}^*+M^*(-i).
\]
Again, it follows from~\eqref{eq:ZT}, equality $M^*(z)=M(\bar z)$, and~\eqref{eq:MlMz} that the adjoint extension $\eA^*$ corresponds to
\[
\wt{\bf T}^*-M(-i)=\bm{Z}^*+M^*(-i)-M(-i)=\bm{Z}^*+2i\gamma^*(i)\gamma(i).\qedhere
\]
\end{proof}


\section{$m$-sectorial extensions of a symmetric operator in the model of two point interactions on a plane}

Let $y_1,y_2\in\dR^2$. Consider in the Hilbert space $L_2(\mathbb{R}^2)$ the operator $A$
given by:
\begin{equation}\label{OPS2}
\begin{aligned}
\dom(A)&=\left\{f(x)\in W_2^2(\dR^2): f(y_1)=f(y_2)=0,\quad k=1,2\right\},\\
Af&=-\Delta f,
\end{aligned}
\end{equation}
where $x=(x_1,x_2)\in\Rtwo$, $W_2^2(\dR^2)$ is a Sobolev space, and
$$\Delta=\cfrac{\partial^2}{\partial x_1^2}+\frac{\partial^2}{\partial x_2^2}$$
is Laplacian.

The operator $A$ is a densely defined closed nonnegative symmetric
with defect indices $(2,2)$~\cite{Albev}. Such operators are basic
in the models of point interactions~\cite{Albev}. In the case of
one point the corresponding operator
\[
\dom(A_y)=\left\{f(x)\in W_2^2(\dR^2): f(y)=0\right\},\; A_y
f=-\Delta f
\]
 admits a unique nonnegative selfadjoint
extension~\cite{Adamyan,GKMT}, the free Hamiltonian:
\begin{gather*}
\dom(A_F)=W_2^2(\dR^2),\quad A_Ff=-\Delta f,
\end{gather*}
Therefore, $A_y$ has no $m$-sectorial and quasi-selfadjoint
$m$-accretive extensions. All $m$-accretive extensions of $A_y$ have
been described in~\cite{ArlPopMAccExt}. For two and more point
interactions the relation $A_F\ne A_N$ holds~\cite{Adamyan}. In this
section we apply Theorems~\ref{thm:bijective} and~\ref{thm:sect}
for a parametrization of all $m$-sectorial extensions of the
operator $A$.
It is convenient to use the Fourier transform and the momentum
representation of $A$:
\begin{align*} \hat A \hat f(p)&=|p|^2\hat f(p),\\
\dom(\hat A)&=\left\{\hat f(p)\in L_2(\dR^2, dp):\begin{aligned} 1)\,&|p|^2\hat f(p)\in L_2(\dR^2,
dp),\\
2)\,&\int_{\dR^2} \hat f(p) e^{ipy_1}dp=\int_{\dR^2} \hat f(p) e^{ipy_2}dp=0.\end{aligned}\right\}
\end{align*}
For a one-center point interaction this method has been used
in~\cite{ArlPopMAccExt}.
In this paper we omit details in the momentum representation and
present final results in the coordinate representation.

The Friedrichs extension of the operator $A$ is the free Hamiltonian
$A_F$ and $A^{1/2}_F=(-\Delta)^{1/2}$ is a pseudodifferential
operator of the form:
\begin{align*}
\dom(A_{F}^{1/2})&=\domf[A_F]=W_2^1(\dR^2),\\
A^{1/2}_F f(x)&=\frac{1}{4\pi^2}\iint\limits_{\dR^2\times\dR^2}|p|\exp{(i(x-y)p)}f(y)dydp,
\end{align*}
where $W^{k}_2(\dR^2)$, $k=1,2$, are the Sobolev spaces. Note that,
see~\cite{Albev}, the resolvent is of the form
\begin{multline*}
(A_F-\lambda
I)^{-1}f(x)=\frac{i}{4}\int_{\dR^2}H_0^{(1)}(\sqrt{\lambda}|x-y|)f(y)dy,\;
f\in L_2(\dR^2),\\
 \lambda\in\dC\setminus [0,+\infty),\;\IM \sqrt{\lambda}>0,
\end{multline*}
where $H_0^{(1)}(\cdot)$ denotes the Hankel function of first kind
and order zero~\cite{olver2010nist}. It is well known~\cite{Albev} that
\begin{multline*}
\sN_{\lambda}=\left\{\frac{\pi i}{2}\sum_{k=1}^2
H_0^{(1)}(\sqrt{\lambda}|x-y_k|)c_k,\; c_1,c_2\in\dC\right\},\\
\lambda\in\dC\setminus [0,+\infty),\; \IM \sqrt{\lambda}>0
\end{multline*}
is the defect
subspace of $A$, corresponding to $\lambda$.
Therefore, for the linear manifold $\sL$ defined by~\eqref{linl} we have
\begin{multline*}
\sL=W^{1}_2(\dR^2)\dot+\sN_\lambda\\
=\left\{ f(x)+ \frac{\pi i}{2}\sum_{k=1}^2 H_0^{(1)}(\sqrt{\lambda}|x-y_k|)c_k,\; f\in W^{1}_2(\dR^2), \; c_1, c_2\in\dC\right\},
\end{multline*}
where $\lambda$ is a number from $\dC\setminus [0,+\infty)$. Now, let $\cH=\dC^2$ and set
\[
\label{granoper} \Gamma \left(f(x)+\frac{\pi i}{2}\sum_{k=1}^2
H_0^{(1)}(\sqrt{\lambda}|x-y_k|)c_k\right)=\vec
c=\begin{bmatrix}c_1\\c_2\end{bmatrix}\in\dC^2,\;f(x)\in
W_2^1(\dR^2).
\]
Then from the equality
$\overline{H_0^{(1)}(\sqrt{\lambda}|x|)}=H_0^{(2)}(\sqrt{\overline{\lambda}}|x|)$~\cite{olver2010nist}
it follows that
\begin{gather*}
\gamma(\lambda)\vec{c}=\frac{\pi i}{2}\sum_{k=1}^2 H_0^{(1)}(\sqrt{\lambda}|x-y_k|)c_k,\, \vec{c}=\begin{bmatrix}c_1\\c_2\end{bmatrix}\in\dC^2,\\
\gamma^*(\bar\lambda) h(x)=-\frac{\pi
i}{2}\begin{bmatrix}\displaystyle\int\limits_{{\mathbb R}^2} h(x)
H_0^{(2)}
(\sqrt{\lambda}|x-y_1|)dx\\ \displaystyle\int\limits_{{\mathbb R}^2}
h(x) H_0^{(2)}(\sqrt{\lambda}|x-y_2|)dx\end{bmatrix}.
\end {gather*}
Set $r=|y_1-y_2|$,
\[
H(\lambda,r)=H_0^{(1)}(\sqrt{\lambda}r)-H_0^{(1)}(e^{3\pi
i/4}r).
\]
From~\eqref{QFU}, using unitarity of the Fourier transform, one can
derive that the matrix $\cQ(\lambda)$ in the standard basis is of
the form:
\begin{equation*}
\cQ(\lambda)=\frac{\lambda-i}{\lambda+i}\pi\begin{bmatrix}-\ln(\lambda i)&\pi iH(\lambda,r)\\
\pi iH(\lambda,r)&-\ln(\lambda i)\end{bmatrix}.
\end{equation*}
 Hence,
\[
\cQ^*(\lambda)=\frac{\bar\lambda+i}{\bar\lambda-i}\pi\begin{bmatrix}-\ln\left(\frac{\bar\lambda}{i}
\right)&-\pi i\bar{H}(\lambda,r
)\\
-\pi
i\bar{H}(\lambda,r)&-\ln\left(\frac{\bar\lambda}{i}\right)\end{bmatrix}.
\]

Now we will find the subspace $\cD_0$ and the sesquilinear form $\Omega_0[\cdot,\cdot]$ (see
Theorem~\ref{lem:QGlims}).
\begin{multline*}
(\cQ(\lambda)\vec{c},\vec{d})=\frac{\lambda-i}{\lambda+i}\pi\begin{bmatrix}d_1\\d_2\end{bmatrix}^*
\begin{bmatrix}-\ln(\lambda i)&\pi iH(\lambda,r
)\\
\pi iH(\lambda,r
)&-\ln(\lambda i)\end{bmatrix}\begin{bmatrix}c_1\\c_2\end{bmatrix}\\
=\frac{\lambda-i}{\lambda+i}\pi\left( -(c_1\bar{d_1}+c_2\bar{d_2})\ln(\lambda i)\right.\\
+\left.(c_2\bar{d_1}+c_1\bar{d_2})\pi i\left(H_0^{(1)}(\sqrt{\lambda}r)-H_0^{(1)}(e^{3\pi i/4}r)\right)\right).
\end{multline*}
Taking into account the asymptotic behavior~\cite{olver2010nist}
\[H_0^{1}(\lambda)=1+\frac{2i}{\pi}\left(\ln\left(\frac{\lambda}{2}\right)+\gamma\right)+o(\lambda),\quad
\lambda\to 0,\] where $\gamma$ is Euler's constant, we see that
\[
\cD_0:=\left\{e\in\cH:\lim_{\substack{\overline{z\to
0},\\z\in\dC\setminus[0,+\infty)}}\left|(\cQ(z)e,e)_\cH\right|<\infty\right\}
=\left\{\begin{bmatrix}\zeta\\-\zeta\end{bmatrix}\in\dC^2:\zeta\in\dC\right\},
\]
Let
\[
\vec{c}_0 = \begin{bmatrix} 1\\ -1\end{bmatrix}.
\]
Then
\begin{multline*}
\Omega_0[\zeta\vec{c}_0,\eta\vec{c}_0]=-\zeta\bar{\eta}\lim_{\substack{\lambda\rightarrow 0\\\lambda < 0}}(\cQ(\lambda)\vec{c}_0,\vec{c}_0)\\
=\pi\zeta\bar{\eta}\lim_{\substack{\lambda\rightarrow 0\\\lambda < 0}}\left(-2\ln(\lambda i)-2\pi i\big(H_0^{(1)}(\sqrt{\lambda}r)-H_0^{(1)}(e^{3\pi i/4}r)\big)\right)\\
=2\pi\zeta\bar{\eta}\lim_{\substack{\lambda\rightarrow 0\\\lambda < 0}}\left(-\ln(\lambda i)-\pi i\Big(1+\frac{2i}{\pi}\big(\ln\left(\frac{\sqrt{\lambda}r}{2}\right)+\gamma\big)-H_0^{(1)}(e^{3\pi i/4}r)\Big)\right)\\
=2\pi\zeta\bar{\eta}\lim_{\substack{\lambda\rightarrow 0\\\lambda < 0}}\left(-\ln(\lambda i)-\pi i+2\ln\left(\frac{\sqrt{\lambda}r}{2}\right)+2\gamma+\pi i H_0^{(1)}(e^{3\pi i/4}r)\right)\\
=4\pi\zeta\bar{\eta}\left(\ln\frac{r}{2} - \frac{3\pi i}{4}+\gamma+\frac{\pi i}{2}H_0^{(1)}(e^{3\pi i/4}r)\right)=\omega_0\cdot\zeta\bar{\eta},
\end{multline*}
where
\[\omega_0=4\pi\left(\ln\frac{r}{2} - \frac{3\pi i}{4}+\gamma+\frac{\pi i}{2}H_0^{(1)}(e^{3\pi i/4}r)\right).\]
From the latter equality one can obtain that
\[
\RE\Omega_0[\zeta\vec{c}_0,\eta\vec{c}_0] = \RE \omega_0\cdot\zeta\bar{\eta} = 4\pi\left(\ln\frac{r}{2} +\gamma+\kker(r)\right)\zeta\bar{\eta},
\]
where the functions $\kker(\cdot)$ and $\kkei(\cdot)$ are Kelvin functions~\cite[p.268]{olver2010nist}, i.e., the real and imaginary parts of the function $\dfrac{\pi i}{2}H_0^{(1)}(e^{3\pi i/4}(\cdot))$, respectively:
\[
\kker(r) + i\kkei(r)=\dfrac{\pi i}{2}H_0^{(1)}(e^{3\pi i/4}r).
\]

For the operator-functions $\Phi(\lambda)$, $\cG(\lambda)$, $\cQ^*(\bar\lambda)$, and $q(\lambda)$ on $\cD_0=\dom(\Omega_0)$ we have:
\begin{align*}
\Phi(\lambda)X\begin{bmatrix}\zeta\\-\zeta\end{bmatrix}&=\frac{\zeta}{4\pi^2}\iint\limits_{\dR^2\times\dR^2}\frac{|p|}{|p|^2-\lambda}\exp{(i(x-y)p)}g(y)dydp,\\
\cG(\lambda)X\begin{bmatrix}\zeta\\-\zeta\end{bmatrix}&=-\frac{\pi i(\lambda+i)\zeta}{2}\begin{bmatrix}\displaystyle\intRtwo \Phi(\lambda)(f(x)) H_0^{(2)}\left(e^{3\pi i/4}|x-y_1|\right)dx\\ \displaystyle\intRtwo \Phi(\lambda)(f(x)) H_0^{(2)}\left(e^{3\pi i/4}|x-y_2|\right)dx\end{bmatrix},\\
\cQ^*(\bar\lambda)\begin{bmatrix}\zeta\\-\zeta\end{bmatrix}&=\frac{\lambda+i}{\lambda-i}\pi\left(-\ln\left(\frac{\lambda}{i}\right)+\pi i\overline{H(\bar\lambda,|y_1-y_2|)}\right)\begin{bmatrix}\zeta\\-\zeta\end{bmatrix},
\end{align*}
\begin{align*}
q(\lambda)\begin{bmatrix}\zeta\\-\zeta\end{bmatrix}&=\frac{\pi i}{2}\frac{1}{i-\lambda}\zeta\left((i+\lambda)(H_0^{(1)}(\sqrt{\lambda}|x-y_2|)-H_0^{(1)}(\sqrt{\lambda}|x-y_1|))\right.\\
&\qquad\qquad\qquad+2i\left.(H_0^{(1)}(e^{\pi i/4}|x-y_1|)-H_0^{(1)}(e^{\pi i/4}|x-y_2|))\right).
\end{align*}

Now we find the operator $X_0e=i\hat A_{F}^{-1/2}\gamma(i)e,\; e\in\cD_0$ from Theorem~\ref{lem:QGlims}. As was mentioned above it is convenient to use the momentum representation. Let $\hat\gamma(\lambda)=\cF\gamma(\lambda)$, where
\[
\hat f(p)=(\cF f)(x)=\cfrac{1}{2\pi}\intRtwo f(x)e^{-ix\cdot p}dx,\quad p=(p_1,p_2).
\]
is the Fourier transform of $f(x)\in L_2(\dR^2,dx)$. Then
\[
\hat\gamma(\lambda)\vec{c}=\sum_{k=1}^{2} c_k\frac{e^{-ipy_k}}{|p|^2-\lambda},\quad\forall \vec{c}=\begin{bmatrix}c_1\\c_2\end{bmatrix}\in\dC^2.
\]
Hence,
\[
\hat X_0\begin{bmatrix}\zeta\\-\zeta\end{bmatrix}=\cF X_0\begin{bmatrix}\zeta\\-\zeta\end{bmatrix}=i\hat
A_{F}^{-1/2}\gamma(i)\begin{bmatrix}\zeta\\-\zeta\end{bmatrix}=\frac{i(e^{-ipy_1}-e^{-ipy_2})}{|p|(|p|^2-i)}\zeta,
\]
So, $\hat X_0\begin{bmatrix}\zeta\\-\zeta\end{bmatrix}=\hat g_0(p)\zeta$, where
\begin{equation}
\label{g0f}
\hat g_0(p)=\dfrac{i(e^{-ipy_1}-e^{-ipy_2})}{|p|(|p|^2-i)}.
\end{equation}
 Getting back to the coordinate representation, we obtain,
 using~\cite{sneddon1995fourier},~\cite[p.671]{GradRyzh}, that
\begin{multline*}
g_0(x)=\cF^{-1}\hat g_0(p)=\cfrac{1}{2\pi}\intRtwo \frac{i(e^{ip(x-y_1)}-e^{ip(x-y_2)})}{|p|(|p|^2-i)} dp=\\
=i\int_0^{+\infty}\frac{J_0(\rho|x-y_1|)-J_0(\rho|x-y_2|)}{\rho^2-i}d\rho=\\
=\frac{\pi i}{2\sqrt{-i}}\left(I_0(\sqrt{-i}|x-y_1|)-L_0(\sqrt{-i}|x-y_1|)\right)\\
-\frac{\pi i}{2\sqrt{-i}}\left(I_0(\sqrt{-i}|x-y_2|)-L_0(\sqrt{-i}|x-y_2|)\right)=\\
=-\frac{\pi}{2}e^{3\pi i/4}\left(\mathbf{M}_0(e^{-\pi i/4}|x-y_1|)-\mathbf{M}_0(e^{-\pi i/4}|x-y_2|)\right),
\end{multline*}
where $I_0(\cdot)$ is the Bessel function and $L_0(\cdot), \mathbf{M}_0(\cdot)$ are modified Struve functions~\cite[p.288]{olver2010nist}. So,
\[
X_0\begin{bmatrix}\zeta\\-\zeta\end{bmatrix}=g_0(x)\zeta,
\]
where
\[
g_0(x)=-\frac{\pi}{2}e^{3\pi i/4}\left(\mathbf{M}_0(e^{-\pi i/4}|x-y_1|)-\mathbf{M}_0(e^{-\pi i/4}|x-y_2|)\right).
\]
According to~\eqref{ravomg} we have
\[
\|g_0(x)\|_\Ltwo^2 = \RE \omega_0= 4\pi\left(\ln\frac{r}{2} +\gamma+\kker(r)\right).
\]
\begin{remark}
Since $\|g_0(x)\|_\Ltwo^2=\|\hat g_0(p)\|_\Ltwo^2$ (the unitarity of
the Fourier transform), expression~\eqref{g0f} for $\hat g_0(p)$
gives
\[
\|\hat g_0(p)\|_\Ltwo^2 = 4\pi\int_0^\infty
\frac{1-J_0(r\rho)}{\rho(\rho^4+1)}d\rho.
\]
On the other hand, due to~\eqref{ravomg}, we have
$$\|g_0(x)\|_\Ltwo^2 = \RE \omega_0.$$
This leads to the value of the improper integral $\displaystyle\int_0^\infty
\frac{1-J_0(r\rho)}{\rho(\rho^4+1)}d\rho$:
\[
\int_0^\infty \frac{1-J_0(r\rho)}{\rho(\rho^4+1)}d\rho=\cfrac{1}{4\pi}\RE \omega_0=\left(\ln\frac{r}{2} +\gamma+\kker(r)\right).
\]
\end{remark}

In order to describe all $m$-sectorial extensions of $A$ we need to define
pairs $\angles{\bm{Z},X}$ satisfying conditions~3), 4) from
Theorem~\ref{thm:main} and conditions~1)--3) of
Theorem~\ref{thm:sect}. Since $\bm{Z}$ is $m$-accretive
linear relation in $\dC^2$ and $\dom(\bm{Z})\subseteq\cD_0$, there are only
two possible cases:
\begin{enumerate}
\item $\bm{Z}=\left<{\begin{bmatrix}\zeta\\-\zeta\end{bmatrix},z\cdot\begin{bmatrix}
\zeta\\-\zeta\end{bmatrix}}\right>\oplus\left<{0,\begin{bmatrix}\eta\\\eta\end{bmatrix}}\right>$,
$\zeta,\eta,z\in\dC$, $\RE z\geqslant 0$;
\item $\bm{Z}=\angles{0,\dC^2}$.
As has been mentioned in~\cite{ArlPopMAccExt} this linear relation
corresponds to the Friedrichs extension $A_F$ of $A$.
\end{enumerate}
In the first case the operator $X,$ acting from $\dom(\bm{Z})$ into ${L_2(\dR^2)},$ takes the form
$X\begin{bmatrix}\zeta\\-\zeta\end{bmatrix}=\zeta g(x),$
where a function $g(x)\in L_2(\dR^2)$ satisfies the condition
\begin{equation}
\label{eq:Zxcond}
\|g(x)\|^2_{L_2(\dR^2)}=\intRtwo |g(x)|^2 dx\leqslant 2\RE z.
\end{equation}
 For the form ${\omega}[\cdot,\cdot]$ defined by~\eqref{formaom} we have
\begin{multline}
{\omega}[\zeta\vec{c}_0,\eta\vec{c}_0]=(\bm{Z}\zeta\vec{c}_0,\eta\vec{c}_0)-\Omega_0[\zeta\vec{c}_0,\eta\vec{c}_0]-2((X-X_0)\zeta\vec{c}_0,X_0\eta\vec{c}_0)\\
=\left(2z-\omega_0-2\intRtwo (g(x)-g_0(x))\overline{g_0(x)} dx\right)\zeta\bar\eta.
\end{multline}
\[
\RE{\omega}[\zeta\vec{c}_0]=\left(2\RE z+\intRtwo \left|g(x)-g_0(x)\right|^2 dx -\intRtwo |g(x)|^2 dx\right)|\zeta|^2.
\]
Thus, the form ${\omega}[\cdot,\cdot]$ is determined by the number
\begin{equation}
 w_{\angles{z, g(x)}} = 2z-\omega_0-2\intRtwo (g(x)-g_0(x))\overline{g_0(x)} dx.
\end{equation}
Clearly, the form ${\omega}[\cdot,\cdot]$ is sectorial iff
\begin{equation}\label{eq:wsectcond}
\begin{gathered}
\RE w_{\angles{z, g(x)}}=2\RE z+\intRtwo \left|g(x)-g_0(x)\right|^2 dx -\intRtwo |g(x)|^2 dx>0\\
\text{ or }w_{\angles{z, g(x)}}= 0.
\end{gathered}
\end{equation}

\begin{remark} Due to
$2\RE z-\intRtwo |g(x)|^2 dx\ge 0$ the equality $ w_{\angles{z, g(x)}} = 0$ implies
that $g(x)=g_0(x)$ almost everywhere and $z=\omega_0/2$.
\end{remark}
Further, condition~3) from Theorem~5 takes the form
\[
M\intRtwo \left|g(x)-g_0(x)\right|^2 dx\leqslant 2\RE z-\intRtwo |g(x)|^2 dx,
\]
where $M>0$. 
The latter inequality 
can be simplified as follows
\begin{equation}
\label{subord}
2\RE z-\intRtwo |g(x)|^2 dx>0.
\end{equation}
So, conditions~\eqref{eq:Zxcond},~\eqref{eq:wsectcond} are satisfied.
Note, that in this case linear relation $\bm{W}(\lambda)$, see~\eqref{eq:lrW}, is the of the form
\begin{multline*}
\bm{W}(\lambda)=\left\langle\begin{bmatrix}\zeta\\-\zeta\end{bmatrix},\left(z-\frac{\lambda+i}{\lambda-i}\pi\left(-\ln\left(\frac{\lambda}{i}\right)+\pi i\overline{H(\bar\lambda,|y_1-y_2|)}\right)\right)\cdot\begin{bmatrix}\zeta\\-\zeta\end{bmatrix}\right.\\
-\left.\pi i(\lambda+i)\zeta\begin{bmatrix}\intRtwo \Phi(\lambda)(g(x)) H_0^{(2)}(e^{3\pi i/4}|x-y_1|)dx\\ \intRtwo \Phi(\lambda)(g(x)) H_0^{(2)}(e^{3\pi i/4}|x-y_2|)dx\end{bmatrix}+\begin{bmatrix}\eta\\\eta\end{bmatrix}\right\rangle.
\end{multline*}
for all $\lambda\in\rho(A_F)=\dC\setminus [0,+\infty)$.
Then
\[
\bm{W}^{-1}(\lambda)=\left\langle\begin{bmatrix}\zeta\\\eta\end{bmatrix}, \frac{1}{w_{\angles{z,g(x)}}(\lambda)}\begin{bmatrix}\zeta-\eta\\-\zeta+\eta\end{bmatrix}\right\rangle,
\]
where
\begin{multline*}
w_{\angles{z,g(x)}}(\lambda)=2\left(z-\frac{\lambda+i}{\lambda-i}\pi\left(-\ln\left(\frac{\lambda}{i}\right)+\pi i\overline{H(\bar\lambda,|y_1-y_2|)}\right)\right)\\
-\pi i(\lambda+i)\intRtwo \Phi(\lambda)(g(x)) \left(H_0^{(2)}(e^{3\pi i/4}|x-y_1|)-H_0^{(2)}(e^{3\pi i/4}|x-y_2|)\right)dx.
\end{multline*}
Clearly, $\ker(\bm{W}(\lambda))\neq \{0\}$ iff $w_{\angles{z, g(x)}}(\lambda)=0$ and
\[
\ker (\bm{W}(\lambda))=\dom (\bm{W}(\lambda))=\begin{bmatrix}\eta\\-\eta\end{bmatrix},\quad \eta\in\dC.
\]
Let an $m$-sectorial extension $\eA$ of $A$ be defined by a pair $\angles{z, g(x)},$ satisfying~\eqref{subord}, see Theorem~\eqref{thm:sect}.
Since $\eA$ is $m$-sectorial extension and
\[
G(-i)=0,\quad \cQ^*(-i)=0,\quad q(-i)=\gamma(i),
\]
it is suitable to take $\lambda=-i$ and apply Theorem~\ref{thm:bijective}, Remark~\ref{mnim}, and equalities~\eqref{domain},~\eqref{action},~\eqref{domain1}.
Then,
\begin{gather*}
\bm{W}=\bm{W}(-i)=\bm{Z}=\left<{\begin{bmatrix}\zeta\\-\zeta\end{bmatrix},z\cdot\begin{bmatrix}\zeta\\-
\zeta\end{bmatrix}}\right>\oplus\left<{\begin{bmatrix}0\cr 0\end{bmatrix},\,\begin{bmatrix}\eta\\\eta\end{bmatrix}}\right>,\\
\bm{W}^{-1}=\left\langle\begin{bmatrix}\zeta\\\eta\end{bmatrix}, \frac{1}{2z}\begin{bmatrix}\zeta-\eta\\-\zeta+\eta\end{bmatrix}\right\rangle.
\end{gather*}
By~\eqref{domain}
\[
\dom(\eA)=\Bigl(I+(q(-i)-2\Phi(-i)X)\bm{W}^{-1}(-i)\gamma^*(-i)(A_F + iI)\Bigr)\dom(A_{F}).
\]

Further, let $\delta(x)$, $x=(x_1,x_2)$ be the Dirac delta. Then $\delta(x)\in W^{-2}_2(\dR^2)$~\cite{Albev}. Since $\cF(\delta(x))=1/2\pi$, then $\cF^{-1}(1)=2\pi\delta(x)$. So, if
$\cF(h(x))=\hat h(p)$ and $h(x)\in\dom (A_F)=W^2_2(\dR^2)$, then
\[
\int_{\dR^2}(e^{ipy_1}-e^{ipy_2})\hat h(p) dp=2\pi (h(y_1)-h(y_2)).
\]
Using the latter equality and the Fourier transform we obtain that
\[
\bm{W}^{-1}(-i)\gamma^*(-i)(A_F+iI)h(x)=\frac{\pi(h(y_1)-h(y_2))}{z}.
\]
If $h(x)\in\dom(A_{F})$, then
\[
\dom(\eA)=\left\{\begin{aligned}&h(x)+\frac{\pi(h(y_1)-h(y_2))}{z}\\
&\quad\times\biggl(\frac{\pi i}{2}\left(H_0^{(1)}(e^{3\pi i/4}|x-y_1|)-H_0^{(1)}(e^{3\pi i/4}|x-y_2|)\right)\\
&\qquad\quad-2\Phi(-i)(g(x))\biggr)\end{aligned}\right\}.
\]

Then applying Theorems~\ref{thm:bijective},~\ref{thm:sect} we arrive at the following statement.
\begin{theorem}
There is a bijective correspondence between all $m$-sectorial
extensions $\eA$ (except Friedrichs and Kre\u\i n-von Neumann
extensions) of $A$ given by~\eqref{OPS2} and all pairs $\angles{z,
g(x)}$, where $z\in\dC$ and a function $g(x)\in\Ltwo$ are such that:
\[
\|g(x)\|^2_{L_2(\dR^2)} < 2\RE z.
\]
This correspondence is given by the relations:
\[
\dom({\widetilde{A}})=\left\{\begin{aligned}
&u(x)=h(x)+\frac{\pi(h(y_1)-h(y_2))}{z}\times\\
&\qquad\quad\times\biggl(\frac{\pi i}{2}\left(H_0^{(1)}(e^{3\pi i/4}|x-y_1|)-H_0^{(1)}(e^{3\pi i/4}|x-y_2|)\right)\\
&\qquad\qquad\quad-2\Phi(-i)(g(x))\biggr),\\
&h(x)\in W_2^2(\Rtwo)\end{aligned}\right\},
\]
\begin{multline*}\label{eq:saction}
\tilde{A} u(x)=-\Delta h(x)-i\frac{\pi(h(y_1)-h(y_2))}{z}\\
\times\left(\frac{\pi i}{2}\left(H_0^{(1)}(e^{3\pi i/4}|x-y_1|)-H_0^{(1)}(e^{3\pi i/4}|x-y_2|)\right)-2\Phi(-i)(g(x))\right).
\end{multline*}
Moreover,
\begin{enumerate}
\item a number $\lambda\in\dC\setminus[0,+\infty)$ is a regular point of $\eA$ if and only if ${w_{\angles{z, g(x)}}(\lambda)\neq 0}$ and,
\begin{multline*}
(\eA-\lambda I)^{-1} h(x)=\frac{i}{4}\intRtwo H_0^{(1)}(\sqrt{\lambda}|x-y|)f(y) dy+\frac{1}{w_{\angles{z, g(x)}}}\\
\times\left(\frac{\pi i}{2}\frac{1}{i-\lambda}\left((i+\lambda)(H_0^{(1)}(\sqrt{\lambda}|x-y_2|)-H_0^{(1)}(\sqrt{\lambda}|x-y_1|))\right.\right.\\
+2i\left.\left.(H_0^{(1)}(e^{\pi i/4}|x-y_1|)-H_0^{(1)}(e^{\pi i/4}|x-y_2|))\right)-2\Phi(\lambda)(g(x))\right)\\
\times\left(-\frac{\pi i}{2}\right)\intRtwo \left(H_0^{(2)}(\sqrt{\lambda}|x-y_1|)-H_0^{(2)}(\sqrt{\lambda}|x-y_2|)\right)h(x)dx.
\end{multline*}

\item a number $\lambda\in\rho(A_F)$ is an eigenvalue of $\eA$ if and only if ${w_{\angles{z, g(x)}}(\lambda) = 0}$ and,
\begin{multline*}
\ker(\eA-\lambda I)=\left(\frac{\pi i}{2}\frac{1}{i-\lambda}\left((i+\lambda)(H_0^{(1)}(\sqrt{\lambda}|x-y_2|)-H_0^{(1)}(\sqrt{\lambda}|x-y_1|))\right.\right.\\
+2i\left.\left.(H_0^{(1)}(e^{\pi i/4}|x-y_1|)-H_0^{(1)}(e^{\pi i/4}|x-y_2|))\right)-2\Phi(\lambda)(g(x))\right)\eta,\quad \eta\in\dC.
\end{multline*}
\end{enumerate}
\end{theorem}

\begin{corollary}
Let $A$ be given by~\eqref{OPS2}. Then there is a bijective
correspondence between all $m$-accretive quasi-selfadjoint
extensions $\eA$ of $A$ (except Friedrichs and Kre\u\i n-von Neumann
extensions) and all
complex numbers $z\in\dC$ such that:
\[\RE z \ge 2\pi\left(\ln\frac{|y_1-y_2|}{2} +\gamma+\kker(|y_1-y_2|)\right).\]
Moreover, an extension $\eA$ is $m$-sectorial if and only if
\[\RE z > 2\pi\left(\ln\frac{|y_1-y_2|}{2} +\gamma+\kker(|y_1-y_2|)\right),\]
and is nonnegative selfadjoint if and only
if
\[
\IM z = \pi\left(-3\pi+4\kkei(|y_1-y_2|)\right)
\]


The correspondence is given by relations
\begin{equation}\label{eq:spdomain}
\dom(\eA)=\left\{\begin{aligned}
&u(x)=h(x)+\frac{\pi(h(y_1)-h(y_2))}{z}\times\\
&\quad\times\biggl(\frac{\pi i}{2}\left(H_0^{(1)}(e^{3\pi i/4}|x-y_1|)-H_0^{(1)}(e^{3\pi i/4}|x-y_2|)\right)\biggr),\\
&h(x)\in W_2^2(\Rtwo)\end{aligned}\right\},
\end{equation}
\begin{multline}\label{eq:spaction}
\eA u(x)=-\Delta h(x)+\frac{\pi^2(h(y_1)-h(y_2))}{2z}\times\\
\times\left(H_0^{(1)}(e^{3\pi i/4}|x-y_1|)-H_0^{(1)}(e^{3\pi i/4}|x-y_2|)\right).
\end{multline}
Moreover,
\begin{enumerate}
\item a number $\lambda\in\dC\setminus [0,+\infty)$ is a regular point of $\eA$ if and only if
\[{w(z,\lambda)=z-\pi\ln(\lambda i)-\pi^2 i(H_0^{(1)}(\sqrt{\lambda}|y_1-y_2|)-H_0^{(1)}(e^{3\pi i/4}|y_1-y_2|))\neq 0}\]
and,
\begin{align*}
(\eA-\lambda I)^{-1} h(x)&=\frac{i}{4}\intRtwo H_0^{(1)}(\sqrt{\lambda}|x-y|)f(y) dy\\
&+\frac{\pi^2}{8w(z,\lambda)}\times\left(H_0^{(1)}(\sqrt{\lambda}|x-y_1|)-H_0^{(1)}(\sqrt{\lambda}|x-y_2|)\right)\\
&\times\intRtwo \left(H_0^{(2)}(\sqrt{\lambda}|x-y_1|)-H_0^{(2)}(\sqrt{\lambda}|x-y_2|)\right)h(x)dx.
\end{align*}

\item a number $\lambda\in\dC\setminus [0,+\infty)$ is an eigenvalue of $\eA$ if and only if ${w(z, \lambda) = 0}$ and,
\[
\ker(\eA-\lambda I)=\left(H_0^{(1)}(\sqrt{\lambda}|x-y_1|)-H_0^{(1)}(\sqrt{\lambda}|x-y_2|)\right)\eta,\quad \eta\in\dC.
\]
\end{enumerate}
\end{corollary}

\begin{remark}
One can obtain a description of the Kre\u\i n-von Neumann extension
$A_N$ of $A$ from relations~\eqref{eq:spdomain},~\eqref{eq:spaction}
by substituting
\[
2z=\omega_0=4\pi\left(\ln\frac{|y_1-y_2|}{2}-\frac{3\pi i}{4}+\gamma+\frac{\pi i}{2}H_0^{(1)}(e^{3\pi i/4}|y_1-y_2|)\right).
\]
It follows from~\eqref{expsnuvsym} that form $A_N[u,v]$ associated with the
Kre\u\i n-von Neumann extension $A_N$ takes the form
\[
\domf[A_N]=\left\{\begin{aligned}
&u(x)=h(x)+\frac{\pi i}{2}\left(H_0^{(1)}(e^{\pi i/4}|x-y_1|)-H_0^{(1)}(e^{\pi i/4}|x-y_2|)\right)\omega,\\
&h(x)\in W_2^1(\Rtwo),\quad\omega\in\dC\end{aligned}\right\},
\]
and if
\begin{align*}
u(x)&=h_1(x)+\frac{\pi i}{2}\left(H_0^{(1)}(e^{\pi i/4}|x-y_1|)-H_0^{(1)}(e^{\pi i/4}|x-y_2|)\right)\omega_1,\\
v(x)&=h_2(x)+\frac{\pi i}{2}\left(H_0^{(1)}(e^{\pi i/4}|x-y_1|)-H_0^{(1)}(e^{\pi i/4}|x-y_2|)\right)\omega_2,
\end{align*}
where $h_1(x),h_2(x)\in W_2^1(\Rtwo),\quad\omega_1,\omega_2\in\dC$,
then
\begin{multline*}
A_N[u,v]=\intRtwo \nabla h_1(x)\overline{\nabla h_2(x)}dx\\
- \frac{\pi\bar\omega_2}{2}\intRtwo h_1(x)\overline{\left(H_0^{(1)}(e^{\pi i/4}|x-y_1|)-H_0^{(1)}(e^{\pi i/4}|x-y_2|)\right)}dx\\
-\frac{\pi\omega_1}{2}\intRtwo \left(H_0^{(1)}(e^{\pi i/4}|x-y_1|)-H_0^{(1)}(e^{\pi i/4}|x-y_2|)\right)\overline{h_2(x)}dx\\
+4\pi\left(\ln\frac{|y_1-y_2|}{2} +\gamma+\kker |y_1-y_2|\right)\cdot\omega_1\bar\omega_2.
\end{multline*}
\end{remark}

\providecommand{\bysame}{\leavevmode\hbox to3em{\hrulefill}\thinspace}
\providecommand{\MR}{\relax\ifhmode\unskip\space\fi MR }
\providecommand{\MRhref}[2]{%
 \href{http://www.ams.org/mathscinet-getitem?mr=#1}{#2}
}
\providecommand{\href}[2]{#2}

\end{document}